\documentclass[12pt]{amsart}

\usepackage{amsmath,amssymb,amsfonts,amscd}
\usepackage{latexsym}
\usepackage{epic,eepic,epsfig,psfrag}
\usepackage{amscd}
\usepackage{hyperref}
\usepackage{color}
\usepackage{amsbsy}
\usepackage{tensor}
\usepackage{fixmath}
\usepackage{tikz-cd}
\usepackage{mathtools}
\usepackage{float}

\let\oldtocsection=\tocsection
\let\oldtocsubsection=\tocsubsection
\let\oldtocsubsubsection=\tocsubsubsection
\renewcommand{\tocsection}[2]{\hspace{0em}\oldtocsection{#1}{#2}}
\renewcommand{\tocsubsection}[2]{\hspace{2em}\oldtocsubsection{#1}{#2}}
\renewcommand{\tocsubsubsection}[2]{\hspace{4em}\oldtocsubsubsection{#1}{#2}}

\def\g{\gamma}

\def\t{\tau}
\def\i{\iota}

\def\G{\Gamma}

\def\cA{{\mathcal A}}

\def\cD{{\mathcal D}}
\def\cE{{\mathcal E}}

\def\cI{{\mathcal I}}

\def\cO{{\mathcal O}}
\def\cP{{\mathcal P}}

\def\cS{{\mathcal S}}

\newtheorem{dfn}{Definition}[section]
\newtheorem{lem}[dfn]{Lemma}
\newtheorem{prp}[dfn]{Proposition}
\newtheorem{thm}[dfn]{Theorem}
\newtheorem{rmk}[dfn]{Remark}
\newtheorem{cor}[dfn]{Corollary}
\newtheorem{ex}[dfn]{Example}

\begin{document}

\author{Gunnar Carlsson, Benjamin Filippenko}

\title{The space of sections of a smooth function}

\begin{abstract}
Given a compact manifold $X$ with boundary and a submersion $f : X \rightarrow Y$ whose restriction to the boundary of $X$ has isolated critical points with distinct critical values and where $Y$ is $[0,1]$ or $S^1$, the connected components of the space of sections of $f$ are computed from $\pi_0$ and $\pi_1$ of the fibers of $f$. This computation is then leveraged to provide new results on a smoothed version of the evasion path problem for mobile sensor networks: From the time-varying homology of the covered region and the time-varying cup-product on cohomology of the boundary, a necessary and sufficient condition for existence of an evasion path and a lower bound on the number of homotopy classes of evasion paths are computed. No connectivity assumptions are required.
\end{abstract}

\maketitle

\tableofcontents

\newpage

\section{Introduction} \label{sec:introduction}

Given a smooth map $f : X \rightarrow Y$ of compact manifolds, we are interested in the homotopy type of the space of continuous sections 
$$\G f = \{ \delta : Y \rightarrow X \,\, | \,\, f \circ \delta = id_Y \}$$
of $f$. In the case that $f$ is a fiber bundle, there is a well-known obstruction theory for constructing sections of $f$ by moving up the skeleta of a CW-decomposition of $Y$. If in addition $Y$ is contractible, then $\G f$ is homotopy equivalent to the fiber of $f$. The situation when $f$ is not a fiber bundle has received much less attention in the literature. It has interesting behavior even when $Y$ is contractible.

For $Y = [0,1]$ and $Y = S^1$, we establish a complete computation (Theorem~\ref{thm:pinbaby}) of $\pi_0(\G f)$
for a class of smooth maps $f : X \rightarrow Y$ that is generic on the boundary $\partial X$. We call these tame functions (Definition~\ref{dfn:tamefunction}). The computation uses $\pi_1$ and $\pi_0$ of the fibers of $f$. Roughly, a tame function $f$ on a smooth compact $X$ with boundary $\partial X$ is a submersion whose restriction $f|_{\partial X}$ has isolated critical points with distinct critical values. In Remark~\ref{rmk:futurework}, we propose various extensions of this result, e.g.\ a computation of $\pi_k$ of the components of $\G f$, higher dimensional $Y$, and other interesting possibilities.

We apply Theorem~\ref{thm:pinbaby} to a smoothed version of the evasion path problem for mobile sensor networks \textsection \ref{sec:theevasionpathproblem}. A mobile sensor network is a collection of sensors moving continuously in a bounded domain $\cD \subset \mathbb{R}^d$ such that each sensor can detect objects within a fixed radius. The evasion path problem asks for the identification of intruders that follow a continuous path in the domain $\cD$ that is at all times disjoint from the region $C$ covered by the sensors. Intruders should be identified from only topological (e.g.,\ no coordinates) information about the mobile sensor network. This problem has been studied in \cite{deSilvaGhristEvasionsFence} \cite{EvasionAdamsCarlsson} \cite{MR3763757} (see \textsection \ref{subsec:priorwork}).

We introduce a smoothed version of the evasion path problem in \textsection \ref{subsec:idealized} which effectively approximates the mobile sensor network version. In Corollary~\ref{cor:discretization}, we obtain a complete computation of the connected components of the space of evasion paths in terms of the time-varying $\pi_0$ and $\pi_1$ of the uncovered region $X = \cD \setminus C$. In Theorem~\ref{thm:evasiontheorem}, we establish a necessary and sufficient condition for existence of an evasion path and moreover a lower bound on the number of connected components in terms of time-varying (co)homological information about the covered region $C$ and its boundary. The cup product plays a crucial role. Note that Theorem~\ref{thm:evasiontheorem} does not require the time-varying covered region $C_t$ to be connected, unlike all prior necessary and sufficient conditions for existence of an evasion path. In the connected case, Theorem~\ref{thm:evasiontheorem} has the simpler form Corollary~\ref{cor:connected}.

The precise situation in Theorem~\ref{thm:pinbaby} is as follows. Refer to Figure~\ref{fig:examples} for examples. Let $X$ be a smooth compact cobordism between manifolds with boundary $X_0$ and $X_1$ (Definition~\ref{dfn:cobordism}). Let $f : X \rightarrow [0,1]$ be a tame function (Defintion~\ref{dfn:tamefunction}), i.e.\ a smooth submersion with $f^{-1}(i) = X_i$ for $i = 0,1$ such that the restriction $f|_{\partial X} : \partial X \rightarrow [0,1]$ to the boundary $\partial X$ (not including $X_0$ and $X_1$) has isolated critical points with distinct critical values. (There is a similar story for $f : X \rightarrow S^1$ where $X$ is a manifold with boundary). Note that $f$ is submersive if, for example, $X \subset \mathbb{R}^d \times [0,1]$ is a codimension-$0$ embedding and $f$ is the projection onto $[0,1]$, as is the case in the smoothed evasion path problem. A next step in this research is to allow $f$ to have critical points in the interior of $X$, making the conditions $C^{\infty}$-generic.

To state the theorem, choose regular values $s_i$ that interleave the critical values $t_i$ of $f|_{\partial X}$:
$$0 = s_0 < t_1 < s_1 < t_2 < \cdots < t_n < s_n = 1.$$
Set $X_i = f^{-1}(s_i)$ and $X_i^{i+1} = f^{-1}([s_i,s_{i+1}]).$
There is a diagram of spaces where all maps are inclusions of regular level sets into the regular cobordisms between them
\begin{equation} \label{eq:diagramofspaces}
\widetilde{ZX} := \big ( X_0 \hookrightarrow X_0^1 \hookleftarrow X_1 \hookrightarrow \cdots \hookrightarrow X_{n-1}^n \hookleftarrow X_n \big ).
\end{equation}

Since each $X_i^{i+1}$ contains exactly $1$ critical point of $f|_{\partial X}$ on the boundary, we can construct a gradient-like vector field whose flow deformation retracts $X_i^{i+1}$ onto either $X_i$ or $X_{i+1}$ depending on whether the outward pointing normal vector $\eta$ at the boundary critical point in $X_i^{i+1}$ satisfies $df(\eta) > 0$ or $df(\eta) < 0$ (by submersivity of $f$, we have $df(\eta) \neq 0$). We keep track of this information via the assignment
\begin{equation} \label{eq:critpointtypemap}
\partial^+(i) :=
\begin{cases} 
      i & \text{if } df(\eta) > 0\\
      i+1 & \text{if } df(\eta) < 0,
\end{cases}
\end{equation}
for $i = 0,\ldots,n-1$. The inclusion $X_{\partial^+(i)} \hookrightarrow X_i^{i+1}$ is a homotopy equivalance. After applying $\pi_0$ to the diagram \eqref{eq:diagramofspaces}, one of the induced maps going into $\pi_0(X_i^{i+1})$ is a bijection, so we obtain a diagram where the arrows point either left or right,
$$\pi_0(ZX) := \big ( \pi_0(X_0) \leftrightarrow \pi_0(X_1) \leftrightarrow \cdots \leftrightarrow \pi_0(X_n) \big ).$$

The following theorem is proved in \textsection \ref{subsec:connectedcomponents}; see Theorems~\ref{thm:computationconnectedcomponents}, \ref{thm:computationcomponentsS1} for the precise statements.

\newpage

\begin{thm} \label{thm:pinbaby} \text{}
\begin{enumerate}

\item There is a surjection
$$\Pi_0 : \pi_0(\Gamma f) \rightarrow \varprojlim \pi_0(ZX)$$
with fiber over $\Psi \in \varprojlim \pi_0(ZX)$ characterized as follows. Let $\mathfrak{b} \in \G f$ such that $\Pi_0(\mathfrak{b}) = \Psi$.  Then $\Pi_0^{-1}(\Psi)$ is naturally in bijection with the orbits of an action of the group $\prod_{i=0}^n \pi_1(X_i,\mathfrak{b}(s_i))$ on the set $\prod_{i=0}^{n-1} \pi_1(X_{\partial^+(i)}, \mathfrak{b}(s_{\partial^+(i)}))$.

\item Let $X$ be a smooth compact manifold with boundary equipped with a submersion $f : X \rightarrow S^1$ whose restriction to the boundary $f|_{\partial X} : \partial X \rightarrow S^1$  has isolated critical points with distinct critical values. Define the diagram $\widetilde{ZX}$ as in \eqref{eq:diagramofspaces} with the additional identity map $X_0 = X_n$, and similarly define $\pi_0(ZX)$. Then the statements in (i) hold.
\end{enumerate}
\end{thm}

\begin{figure}
\begin{tikzpicture}
    \matrix[matrix of nodes]{
    \includegraphics[scale=.3]{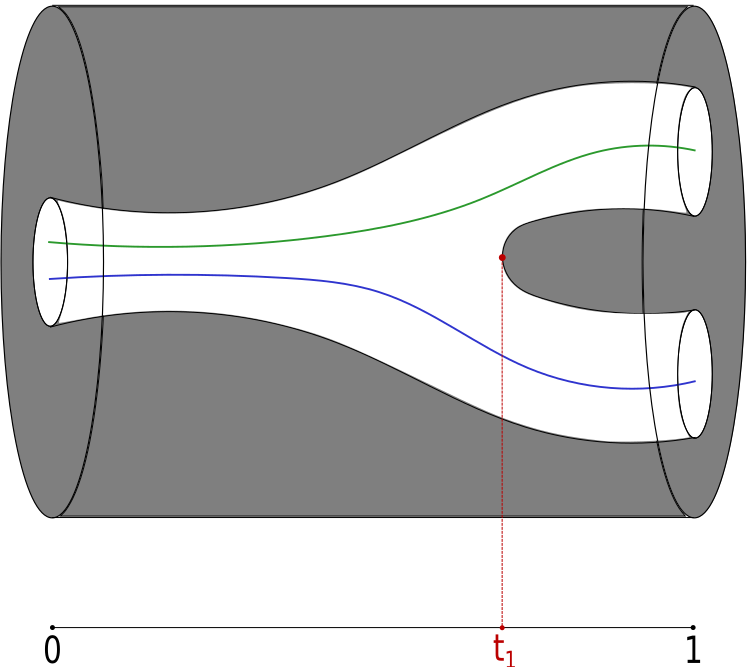} & \includegraphics[scale=.3]{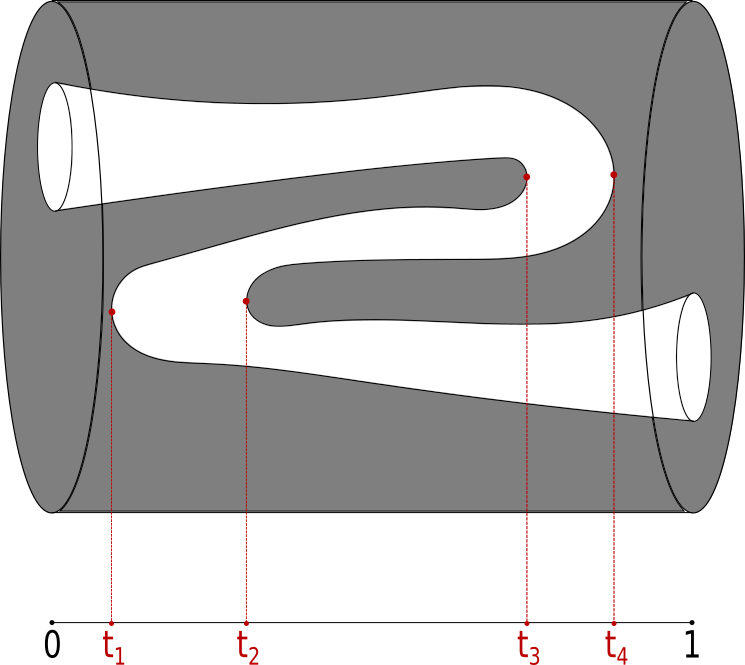}\\
    (a) & (b)\\
    \text{} & \text{}\\
        };
  \end{tikzpicture}
  \begin{tikzpicture}
  \matrix[matrix of nodes]{
    \includegraphics[scale=.3]{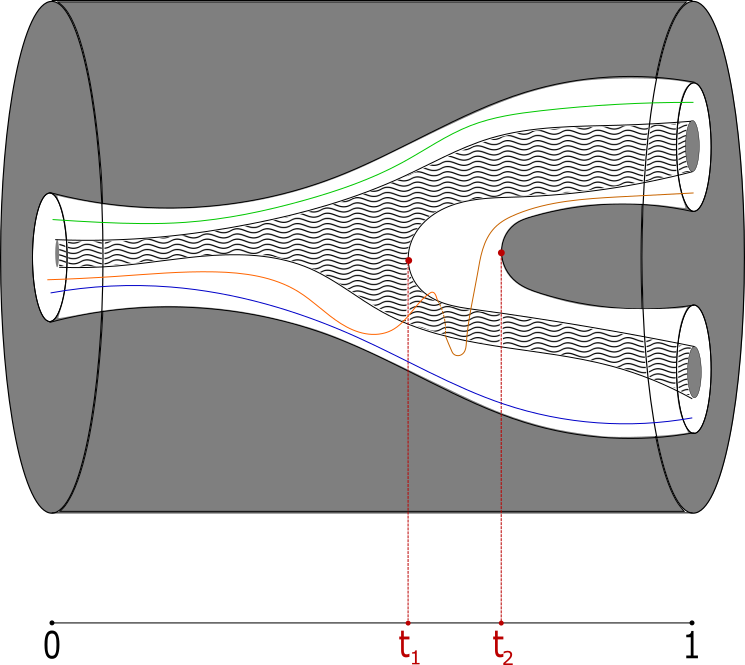}  \\
    (c) \\
 };
    \end{tikzpicture}
\caption{The $3$-dimensional white cobordisms $X$ are embedded in the grey $D^2 \times [0,1]$ and the projections $f$ to $[0,1]$ are tame. The wavy shaded region in $(c)$ is not in $X$.  The critical points of the restriction $f|_{\partial X} : \partial X \rightarrow [0,1]$ to the boundary $\partial X$ are the red dots with critical values $t_i$. Sections of $f$ are shown in blue, green, and, orange.}
\label{fig:examples}
\end{figure}

\begin{rmk} \label{rmk:applytheoremtoexamples}
Applying Theorem~\ref{thm:pinbaby} to the examples in Figure~\ref{fig:examples} produces the following.

In $(a)$, $X$ is an embedded solid pair of pants  and the sections $\G f$ of the projection $f : X \rightarrow [0,1]$ have two connected components $|\pi_0(\G f)| = 2$, represented by the green and blue sections.

In $(b)$, $X$ is an embedded solid cylinder and there are no sections, $\G f = \emptyset$.

In (c), $X$ is a cobordism between a single annulus and two annuli. There are infinitely many connected components of the sections $|\pi_0(\G f)| = \infty$. Indeed, the three sections pictured are not fiberwise homotopic, and one can modify the orange section to wrap around the bottom arm of the wavy region any integral number of times, producing a countable collection of non-fiberwise homotopic sections.
\end{rmk}

\begin{rmk} \label{rmk:futurework}
We propose the following extensions of these results to a general theory of sections of $C^{\infty}$-generic smooth maps $f : X \rightarrow Y$.
\begin{enumerate}
\item There is a generalization of Theorem~\ref{thm:pinbaby} to a computation of $\pi_k(\G f , b)$ for any basepoint $b \in \G f$ and $k \geq 1$. These higher $\pi_k$ can be addressed with a fiberwise version of the unstable Adams spectral sequence of Bousfield-Kan. This is currently being pursued by Wyatt Mackey.
\item What happens when we allow $f$ to have critical points in the interior of $X$? Answering this will complete the $\dim Y = 1$ story for $C^{\infty}$-generic $f$ (isolated critical points with distinct critical values).
\item Is there a sheaf theoretic interpretation of these results? The regular level sets in the theorem are homotopy equivalent to preimages of small open neighborhoods, so the diagram $\widetilde{ZX}$ in \eqref{eq:diagramofspaces} comes from an open covering of $[0,1]$ and inclusions of intersections of those open sets. This suggests replacing $\pi_0(ZX)$ with the co-presheaf $U \mapsto \pi_0(f^{-1}(U))$ on $[0,1]$.
\item Generalize Theorem~\ref{thm:pinbaby} to higher dimensions $\dim Y > 1$. We expect that the $\pi_1$-actions will generalize to $\pi_{\dim Y}$-actions.
\end{enumerate}
\end{rmk}

{\bf The Evasion Path Problem:} We apply Theorem~\ref{thm:pinbaby} to obtain new results on the evasion path problem in applied topology: Theorem~\ref{thm:evasiontheorem} and Corollary~\ref{cor:connected}. The evasion path problem is summarized as follows; see \textsection \ref{sec:theevasionpathproblem} for a detailed description. Given a collection of continuous sensors $\cS = \{ \g : [0,1] \rightarrow \cD \}$ moving in a bounded domain $\cD \subset \mathbb{R}^d$ that detect objects within some fixed radius of $\g(t)$ in $\mathbb{R}^d$, an \emph{evasion path} is a continuous intruder $\delta : [0,1] \rightarrow \cD$ that avoids detection by the sensors for the whole time interval $I = [0,1]$. Let $C_t \subset \cD$ denote the region covered by the sensors at time $t \in I$, and set
$$X_t = \cD \setminus C_t.$$
Then evasion paths $\delta$ are sections of the function
$$f : X := \bigcup_{t \in I} X_t \times \{t\} \subset \mathbb{R}^d \times I \longrightarrow I.$$ 
Let $\G f$ denote the space of evasion paths.

The \emph{sensor ball evasion path problem} asks for a criterion that determines whether or not an evasion path exists and that is based only on homological information about the covered region $C_t$. In practice, one imagines that we can understand the topology of $C_t$ since it is the region covered by the sensors. For example, if sensors can detect overlaps of their sensed regions, then \v{C}ech cohomology of $C_t$ can be computed. Versions of this problem have been studied in \cite{deSilvaGhristEvasionsFence} \cite{EvasionAdamsCarlsson} \cite{MR3763757}; see \ref{subsec:priorwork}.

We consider an idealized version of the evasion path problem; see \textsection \ref{subsec:idealized}. Roughly, we smooth $C =  \bigcup_{t \in I} C_t \times \{t\} \subset \mathbb{R}^d \times I$ into a smooth cobordism of manifolds with boundary embedded in $\mathbb{R}^d \times I$ that closely approximates the region covered by the sensors and whose projection $C \rightarrow I$ is tame.
 
Let $B$ denote the boundary of $C$, except for the interior of the fibers over $0$ and $1$. Note that both $B$ and $C$ have associated diagrams $\widetilde{ZB}$ and $\widetilde{ZC}$ defined in the same way as $\widetilde{ZX}$ in \eqref{eq:diagramofspaces}.

The precise statement of Theorem~\ref{thm:evasiontheorem} is given in Theorem~\ref{thm:main}. In Example~\ref{ex:homologicaltheorem}, Theorem~\ref{thm:evasiontheorem} is applied to example (a) from Figure~\ref{fig:examples}. See Remark~\ref{rmk:mainlowdims} for the $d=0,1$ cases.

\begin{thm} \label{thm:evasiontheorem}
Assume $d \geq 2$. There is a surjection $\pi_0(\G f) \rightarrow \varprojlim Hom_{k-algebra}( H^0(\widetilde{ZX};k), k)$. In particular, an evasion path exists (i.e.\ $\G f$ is nonempty) if and only if $\varprojlim Hom_{k-algebra}( H^0(\widetilde{ZX};k), k )$ is nonempty, and the cardinality of $\pi_0(\G f)$ is bounded from below by the cardinality of the inverse limit.

Assume that the projection $C \rightarrow I$ does not have any local minima or local maxima except over $0,1 \in I$, which holds if sensors are not created or destroyed in time. Then the zigzag diagram of $k$-algebras $H^0(\widetilde{ZX};k)$ is determined up to isomorphism by the zigzag diagram of $k$-algebras $H^0(\widetilde{ZB};k)$, the map $\partial^+$ defined in \eqref{eq:critpointtypemap}, and an Alexander duality isomorphism of $H^0$ of the regular fibers of $B$ with $H_{d-1}$ of their complements $B^c$ as well as maps on $H_{d-1}$ induced fiberwise by inclusion $C \rightarrow B^c$.
\end{thm}

\begin{cor} \label{cor:connected}
Assume that $C_t$ is connected for all $t \in I$. Then there is a surjection $\pi_0(\G f) \rightarrow \varprojlim Hom_{k-algebra}( H^0(\widetilde{ZB};k), k)$. In particular, an evasion path exists (i.e.\ $\G f$ is nonempty) if and only if $\varprojlim Hom_{k-algebra}( H^0(\widetilde{ZB};k), k )$ is nonempty, and the cardinality of $\pi_0(\G f)$ is bounded from below by the cardinality of the inverse limit.
\end{cor}
\begin{proof}
This follows immediately from Corollary~\ref{cor:discretization}, Proposition~\ref{prp:dualizeH0algebra}, and Proposition~\ref{prp:computationintermsofcoveredandboundary}.
\end{proof}

\begin{rmk}
The following heuristic suggests that in applications to sensor networks it is enough to compute $\varprojlim Hom_{k-algebra}( H^0(\widetilde{ZX};k), k)$ to understand all evasion paths that an intruder is likely to take. The fibers of the surjection $\pi_0(\G f) \rightarrow \varprojlim Hom_{k-algebra}( H^0(\widetilde{ZX};k), k)$ are described by $\pi_1$ of the regular fibers of the map $X \rightarrow I$, as in Theorem~\ref{thm:pinbaby}. An intruder is randomly sampling from the space of paths, and so is unlikely to wrap around nontrivial loops in the regular fibers.
\end{rmk}

{\bf Funding.} This material is based upon work supported by the National Science Foundation under Award No.~1903023.

\section{Tame functions on cobordisms of manifolds with boundary} \label{sec:tamefunctions}

\begin{dfn} \label{dfn:cobordism}
A {\bf compact cobordism of manifolds with boundary} is a compact $(d+1)$-dimensional manifold $X$ with boundary and corners\footnote{For $0 \leq k \leq n$, the $k$-stratum $\partial_kX$ of a $n$-dimensional manifold $X$ with boundary and corners consists of those points $x \in X$ around which there is a smooth chart to $[0,\infty)^k \times \mathbb{R}^{n-k}$ that identifies $y$ with a point in $\{0\}^k \times \mathbb{R}^{n-k}$. For a cobordism $X$ of manifolds with boundary, the highest nonempty stratum is $\partial_2 X = \partial B = \partial X_0 \sqcup \partial X_1$. The $1$-stratum $\partial_1 X$ is the union of the interiors of $X_1, B,$ and $X_2$, which together with the $2$-stratum forms the full boundary $\partial X$. The $0$-stratum $\partial_0 X$ is the interior of $X$.} whose boundary decomposes as a union
$$\partial X = X_- \cup B \cup X_+$$
where $B$ and each $X_{\pm}$ is an embedded $d$-dimensional manifold with boundary and such that the following properties hold:
\begin{itemize}
\item The intersection $X_- \cap X_+ = \emptyset$ is empty. We say that $X$ is a {\bf cobordism between $X_-$ and $X_+$},\\
\item $\partial X_{\pm} = X_{\pm} \cap B,$\\
\item $B$ is a cobordism between the closed manifolds $\partial X_-$ and $\partial X_+,$ i.e., $\partial B = \partial X_- \sqcup \partial X_+.$
\end{itemize}
\end{dfn}

The following notion of a tame function, as well as the constructions we perform with them in this section, is inspired by the Morse theory on manifolds with boundary; see \cite{MR3203357}\cite{MR3493413}\cite{MR348790}\cite{MR611759}\cite{MR0331409}\cite{MR2388043}\cite{MR2819662}\cite{MR0190942} and Remark~\ref{rmk:tameversusmorse}.

\begin{dfn} \label{dfn:tamefunction} A {\bf tame function} $f : X \rightarrow [s_-,s_+]$ on a compact cobordism of manifolds with boundary is a smooth function satisfying the following conditions:
\begin{itemize}
\item The critical points of the restriction $f|_B : B \rightarrow [s_-,s_+]$ are isolated and have distinct critical values in $(s_-,s_+)$,\\
\item $f$ is submersive,\\
\item $X_- = f^{-1}(s_-)$ and $X_+ = f^{-1}(s_+)$.
\end{itemize}
\end{dfn}

\begin{rmk} \label{rmk:tameversusmorse}
Definition~\ref{dfn:tamefunction} does not require any nondegeneracy condition on the critical points of $f|_B : B \rightarrow [s_-,s_+]$, so $f|_B$ does not have to be a Morse function in the sense of \cite[Def.~2.3]{MR0190942}. In this way, our definition is more general than the Morse condition. On the other hand, we do not allow $f$ itself to have critical points. Note that for the projection $f : X \rightarrow [s_-,s_+]$ from a codimension-$0$ submanifold $X \subset \mathbb{R}^d \times [s_-,s_+]$, as is the situation in the smoothed evasion path problem (see \textsection\ref{sec:theevasionpathproblem}), the map $f$ is submersive.
\end{rmk}

There are two types of critical points of a tame function on the boundary.

\begin{dfn} \label{dfn:criticalpointtypes}
Let $f : X \rightarrow [s_-,s_+]$ be a tame function, $p \in B$ a critical point of $f|_B : B \rightarrow [s_-,s_+]$, and $\eta \in T_pX$ an outward pointing vector. Then $p$ is {\bf type N} if $df(\eta) < 0$ and {\bf type D} if $df(\eta) > 0.$
\end{dfn}

Note that $df(\eta) = 0$ is impossible since it would imply that $p$ is a critical point of $f$.

\subsection{Local constructions}
The following constructions are local in the sense that for a tame function $f : X \rightarrow [s_-,s_+]$ and $t \in [s_-,s_+$], the constructions apply in the preimage $f^{-1}([t-\epsilon,t+\epsilon])$ for $\epsilon > 0$ small enough.

To construct sections of a tame function that pass a critical value, we make use of the flow lines of the gradient-like vector field $\xi$ on $X$ constructed in the following.

\begin{prp} \label{prp:pseudogradient}
Let $f : X \rightarrow [s_-,s_+]$ be a tame function.

If $f|_B$ has only type D critical points, then there exists a smooth vector field $\xi$ on $X$ with the following properties:
\begin{enumerate}
\item $df(\xi) = -1$ on all of $X$.\\
\item For all\footnote{For $p \in X_-$, the vector field $\xi$ constructed in the proof of Proposition~\ref{prp:pseudogradient} is outward pointing on $Int(X_-)$, and on $\partial X_-$ it is outward pointing with respect to $X_-$ and inward pointing with respect to $B$.} $p \in \partial X \setminus X_-$, the vector $\xi_p \in T_pX$ is inward pointing.\\
\end{enumerate}

If $f|_B$ has only type N critical points, then there exists a smooth vector field $\xi$ on $X$ with the following properties:
\begin{enumerate}
\item $df(\xi) = 1$ on all of $X$.\\
\item For all $p \in \partial X \setminus X_+$, the vector $\xi_p \in T_pX$ is inward pointing.\\
\end{enumerate}

Moreover, given any smooth section $\mathfrak{b} : [s_-,s_+] \rightarrow X$ of $f$ (i.e.\, $f \circ \mathfrak{b}(t) = t$) such that $\mathfrak{b}(t) \not \in B$ for all $t \in [s_-,s_+]$, the vector field $\xi$ can be chosen such that
$$\frac{d}{dt} \mathfrak{b} (t) =
\begin{cases} 
      -\xi|_{\mathfrak{b}(t)}  & \text{if } \text{type D} \\
     \xi|_{\mathfrak{b}(t)}  & \text{if } \text{type N}.
   \end{cases}
$$
\end{prp}
\begin{proof}
We consider the case of type D critical points; the type N case is symmetric.

It suffices to construct $\xi$ locally in an open neighborhood of every point $p \in X$ and then sum up these local vector fields with a partition of unity. For the local construction, let $p \in X$ and for now assume that if $p \in B$ then $p$ is a regular point of $f|_B$. Then by the implicit function theorem there exists an open neighborhood $U(p) \subset X$ of $p$ and a smooth chart $\varphi : U(p) \xrightarrow{\sim} V$ such that $\varphi(p) = (f(p),0,\ldots,0) \in V$ where $V$ is an open subset of one of the following spaces depending on where $p$ sits on $X$, and in such a way that
$$f \circ \varphi^{-1} : V \rightarrow [s_-,s_+]$$
is the projection onto the first coordinate:
\begin{itemize}
\item If $p \in Int(X)$, then $V \subset \mathbb{R}^{\dim X}$,\\
\item If $p \in Int(X_-)$, then $V \subset [s_-,\infty) \times \mathbb{R}^{\dim X -1}$,\\
\item If $p \in Int(X_+)$, then $V \subset (-\infty,s_+] \times \mathbb{R}^{\dim X -1}$,\\
\item If $p \in Int(B)$, then $V \subset \mathbb{R} \times ([0,\infty) \times \mathbb{R}^{\dim X -2})$,\\
\item If $p \in \partial X_-$, then $V \subset [s_-,\infty) \times ([0,\infty) \times \mathbb{R}^{\dim X -2})$,\\
\item If $p \in \partial X_+$, then $V \subset (-\infty,s_+] \times ([0,\infty) \times \mathbb{R}^{\dim X -2})$.\\
\end{itemize}
In all cases above, the constant vector field
$$(-1,1,0,\ldots,0)$$
on $V$ pulls back through the diffeomorphism $\varphi$ to a vector field $\xi$ on $U(p)$ satisfying $df(\xi) = -1$, and moreover it is inward pointing along all points $p \in \partial X \setminus X_- = Int(B) \cup X_+$, as required.

It remains to consider a critical point $p \in B$ of type $D$. By definition of tame function, $p$ is in $Int(B) = B \setminus (\partial X_- \cup \partial X_+)$. Consider a neighborhood $\tilde{U}(p) \subset X$ of $p$ and a coordinate chart $\tilde{U}(p) \xrightarrow{\sim} V \subset [0,\infty) \times \mathbb{R}^{\dim X -1}$ that sends $p$ to $0$. Then the constant vector field $(1,0,\ldots,0)$ pulls back to a vector field $\xi$ on $\tilde{U}(p)$ that is inward pointing, and hence $d_pf(\xi_p) < 0$ since $p$ is type D. Then $df(\xi) < 0$ in a smaller open neighborhood $U(p) \subset \tilde{U}(p)$. Hence the vector field $\xi / |df(\xi)|$ satisfies (i) and (ii) on $U(p)$, as required.

Consider now the final statement of the proposition where we are given a smooth section $\mathfrak{b}$ that is disjoint from $B$. For $p \neq \mathfrak{b}(t)$ for all $t$, choose the neighborhood $U(p)$ to be disjoint from the image of $\mathfrak{b}$ and perform the construction as above. Suppose $p = \mathfrak{b}(t)$ for some $t$. Choose $U(p)$ to be disjoint from $B$. Define $\xi|_{\mathfrak{b}(t)} = -\mathfrak{b}'(t)$ for all $t$ such that $\mathfrak{b}(t) \in U(p)$. Then $df(\xi|_{\mathfrak{b}(t)}) = -(f \circ \mathfrak{b})'(t) = -1$. Extend $\xi$ over $U(p)$ so that $df(\xi) = -1$ on $U(p)$.
\end{proof}

Using the flow of $\xi$, we now construct a deformation retraction of $X$ onto either $X_-$ or $X_+$ that move the values of $f: X \rightarrow [s_-,s_+]$ at constant speed $\pm 1$ along $[s_-,s_+]$. In particular, fibers get mapped to fibers at all times throughout the deformation retraction. This deformation retraction is a workhorse that is used throughout.

\begin{lem} \label{lem:defretracttoregvalues}
Let $f : X \rightarrow [s_-,s_+]$ be a tame function.

If all critical points of $f|_B$ are type $D$, then $X$ deformation retracts onto $X_- = f^{-1}(s_-)$. Moreover, there is a deformation retraction 
$$H : X \times [0, s_+-s_-] \rightarrow X$$
that moves the fibers of $f$ along $[s_-,s_+]$ at constant speed $-1$ until they reach $X_-$ and stop; precisely, $H$ satisfies
$$f(H(x,t)) = \max\{s_-, \,\, f(x) - t\}$$
and
$$H(H(x,t),\Delta) = H(x,t+\Delta)$$
for all $\Delta$.
Moreover, given any smooth section $\mathfrak{b} : [s_-,s_+] \rightarrow X$ of $f$ (i.e.\, $f \circ \mathfrak{b}(t) = t$) disjoint from $B$, the deformation retraction $H$ can be chosen so that
$$H(\mathfrak{b}(t),\Delta) = \mathfrak{b}(t-\Delta)$$
for all $t \in [s_-,s_+]$ and $0 \leq \Delta \leq t - s_-$.

Symmetrically, if all critical points of $f|_B$ are type $N$, then there is a deformation retraction $H : X \times [0,s_+-s_-] \rightarrow X$ of $X$ onto $X_+$ satisfying
$$f(H(x,t)) = \min\{s_+, \,\, f(x) + t\}$$
and
$$H(H(x,t),\Delta) = H(x,t+\Delta)$$
for all $\Delta$. Moreover, given any smooth section $\mathfrak{b} : [s_-,s_+] \rightarrow X$ of $f$ disjoint from $B$, the deformation retraction $H$ can be chosen so that
$$H(\mathfrak{b}(t),\Delta) = \mathfrak{b}(t+\Delta)$$
for all $t \in [s_-,s_+]$ and $0 \leq \Delta \leq s_+ - t$.
\end{lem}
\begin{proof}
Assume that all critical points of $f|_B$ are type D; the type N case is symmetric. Let $\xi$ be a vector field on $X$ having the properties guaranteed by Proposition~\ref{prp:pseudogradient}.

We claim that, for $x \in X$, the flow $\varphi_t(x)$ of $\xi$ exists for all $0 \leq t \leq f(x) - s_-$. Indeed, let $\gamma_x : [0,r] \rightarrow X$ be the flow line of $\xi$ starting at
$$\gamma_x(0) = x.$$
We have $\frac{d}{dt}(f \circ \gamma_x) = df(\xi) = -1$ everywhere in the domain of $\gamma_x$ and hence
$$f \circ \gamma_x(t) = f(x) - t.$$
Since $\xi$ is inward pointing along the boundary of $X$ except on $X_-$, it follows that the flow line $\gamma_x(t)$ can be extended to larger $t$ as long as $\gamma_x(t) \not \in X_-$, and if $\gamma_x(t) \in X_-$ then the flow lines stops. Since $\gamma_x(t) \in X_-$ if and only if $s_- = f \circ \gamma_x(t) = f(x) - t$, it follows that $r = f(x) - s_-$. That is, $\gamma_x$ is defined on the interval $[0,f(x)-s_-]$. And, indeed, the flow is given by
$$\varphi_t(x) = \g_x(t) \text{ for all } 0 \leq t \leq f(x) - s_-.$$
The deformation retraction $H$ is defined using the flow $\varphi$ of $\xi$,
\begin{align*}
H : X \times [0,s_+-s_-] &\rightarrow X\\
(x,t) &\mapsto
\begin{cases} 
      \varphi_t(x) & \text{if } 0 \leq t \leq f(x) - s_- \\
      \varphi_{f(x)-s_-}(x) & \text{if } f(x) - s_- \leq t \leq s_+ - s_-.\\
   \end{cases}
\end{align*}
Indeed, $H$ is a deformation retraction of $X$ onto $X_-$ that satisfies the claimed properties.

Given a smooth section $\mathfrak{b} : [s_-,s_+] \rightarrow X$ of $f$, we choose the vector $\xi$ to satisfy $\mathfrak{b}'(t) = -\xi$, as is made possible by Proposition~\ref{prp:pseudogradient}. The flow of $\xi$ is everywhere tangent to $\mathfrak{b}$. Since also $f(\varphi_{\Delta}(\mathfrak{b}(t))) = t - \Delta$, it follows that $\varphi_{\Delta}(\mathfrak{b}(t)) = \mathfrak{b}(t-\Delta)$. Hence, for $t \in [s_-,s_+]$ and $0 \leq \Delta \leq t - s_-$, we have  $H(\mathfrak{b}(t),\Delta) = \varphi_{\Delta}(\mathfrak{b}(t)) = \mathfrak{b}(t-\Delta)$, as claimed.
\end{proof}

We also make use of the classical statement that, near a regular value $t$ of $f|_B : B \rightarrow [s_-,s_+]$, $X$ is diffeomorphic to the fiber $f^{-1}(t)$ times an interval. The precise statement in our setting is the following.

\begin{lem} \label{lem:trivialcobordismnearregularvalue}
Let $f : X \rightarrow [s_-,s_+]$ be a tame function (in particular a submersion) such that its restriction $f|_B : B \rightarrow [s_-,s_+]$ also is a submersion. Then, for every $t \in [s_-,s_+]$, $X$ is `fiberwise diffeormorphic' to the trivial cobordism $[s_-,s_+] \times f^{-1}(t)$. Precisely, this means that there exists a diffeomorphism
$$\varphi : X \xrightarrow{\cong} [s_-,s_+] \times f^{-1}(t)$$
such that
$$f = pr_{[s_-,s_+]} \circ \varphi,$$
where $pr_{[s_-,s_+]} : [s_-,s_+] \times f^{-1}(t) \rightarrow [s_-,s_+]$ is projection. Moreover, the restriction $\varphi|_{f^{-1}(t)} : f^{-1}(t) \rightarrow \{t\} \times f^{-1}(t)$ is the identity on $f^{-1}(t)$.
\end{lem}
\begin{proof} 
The proof is standard. In summary, construct a gradient-like vector field $\xi$ such that $df(\xi) = -1$ and $\xi$ is tangent to the boundary $B$, outward pointing along $X_-$, and inward pointing along $X_+$ (using the same method as the proof of Proposition~\ref{prp:pseudogradient}, or see the proof of \cite[Lem.~3.1]{MR0331409}). Then build the diffeomorphism from the flow of $\xi$ similarly to Lemma~\ref{lem:defretracttoregvalues} (see also the classical reference \cite[Thm.~3.4]{MR0190942}).
\end{proof}

\subsection{Zigzag discretization} \label{subsec:zigzagdiscretization}
Let $f : X \rightarrow I = [0,1]$ be a tame function. Denote the critical values of $f|_B$ by
$$t_1 < t_2 < \cdots < t_n$$
and choose interleaving values
$$0 = s_0 < t_1 < s_1 < t_2 < \cdots < s_{n-1} < t_n < s_n =  1.$$
Consider the subspaces of $X$ given by the preimages
\begin{align*}
X_i &:= f^{-1}(s_i), \text{ for } 0 \leq i \leq n,\\
X_i^{i+1} &:= f^{-1}([s_i,s_{i+1}]), \text{ for } 0 \leq i \leq n-1.
\end{align*}
Since the $s_i$ are regular values of $f$ and $f|_B$, the $X_i$ are manifolds with boundary $\partial X_i = X_i \cap B$ and $X_i^{i+1}$ is a cobordism between $X_i$ and $X_{i+1}$.

These cobordisms and the inclusions of their boundaries fit together into a zigzag diagram of manifolds
\begin{equation*} \label{eq:widezigzagX}
\widetilde{ZX} := \bigg ( X_0 \xhookrightarrow{\eta_{X_0^1}^-} X_0^1 \xhookleftarrow{\eta_{X_0^1}^+} X_1 \xhookrightarrow{\eta_{X_1^2}^-} X_1^2 \xhookleftarrow{\eta_{X_1^2}^+} \cdots \xhookrightarrow{\eta_{X_{n-1}^n}^-} X_{n-1}^n \xhookleftarrow{\eta_{X_{n-1}^n}^+} X_n \bigg ).
\end{equation*}

Each $X_i^{i+1}$ for $i = 0,\ldots,n-1$ contains exactly $1$ critical point of $f|_B$. Suppose this critical point is type $D$. It follows from Lemma~\ref{lem:defretracttoregvalues} that the inclusion $\eta_{X_i^{i+1}}^{-}$ is an inclusion of a deformation retract. Let $H : X_i^{i+1} \times [0,s_{i+1} - s_i] \rightarrow X_i^{i+1}$ be a deformation retraction onto $X_i$ with the properties provided by the lemma. Then we have
$$H(-,s_{i+1}-s_i) \circ \eta_{X_i^{i+1}}^{-} = id_{X_i}.$$
There is a composite map
$$\beta_{i,i+1} := H(-,s_{i+1}-s_i) \circ \eta_{X_i^{i+1}}^{+} : X_{i+1} \rightarrow X_i.$$

If the critical point in $X_i^{i+1}$ is instead type N, then in the same way a map $\beta_{i,i+1}  : X_{i} \rightarrow X_{i+1}$ is obtained. When we do not wish to specify the type of the critical point, we write this map with arrows pointing \emph{both} ways
$$\beta_{i,i+1}  : X_{i} \leftrightarrow X_{i+1}.$$

\begin{rmk} \rm The double arrow $\leftrightarrow$ indicates a map \emph{either} to the right \emph{or} to the left, \emph{not} both ways.
\end{rmk}

To disambiguate when desired, we define the following mappings which record which side of $X_i^{i+1}$ it deformation retracts onto, and hence determines the direction of $\beta_{i,i+1}$. Precisely, for $i = 0,\ldots,n-1$, define
$$
\partial^+(i) :=
\begin{cases} 
      i & \text{if } t_i \text{ is type-}D\\
      i+1 & \text{if } t_i \text{ is type-}N,
\end{cases}
$$
and also for notational convenience define the opposite mapping $$
\partial^-(i) :=
\begin{cases} 
      i+1 & \text{if } t_i \text{ is type-}D\\
      i & \text{if } t_i \text{ is type-}N.
\end{cases}
$$
Then, for both type D and type N critical points, the double arrow denoting $\beta_{i,i+1}$ points in the direction
\begin{equation} \label{eq:betadirected}
\beta_{i,i+1} : X_{\partial^-(i)} \rightarrow X_{\partial^+(i)}.
\end{equation}

The diagram $\widetilde{ZX}$ is then the top `zigzag' of the larger diagram
\[
  \begin{tikzcd}
        & X_0^1 &          & \,\,\, \cdots \,\,\, & & X_{n-1}^n &\\
 X_0  \arrow[hook]{ru}  \arrow{rr}{\beta_{0,1}}  &              & \arrow{ll} \arrow[hook]{lu}  X_1 \arrow{r}{\beta_{1,2}} \arrow[hook]{ru} & \arrow{l} \cdots \arrow{r}{\beta_{n-2,n-1}} & \arrow{l} \arrow[hook]{lu} X_{n-1} \arrow[hook]{ru} \arrow{rr}{\beta_{n-1,n}} &            & \arrow{ll} \arrow[hook]{lu} X_n.
\end{tikzcd}
\]
It is often more convenient to work with the bottom row, which we denote by
\begin{equation} \label{eq:zigzagX}
ZX := (X_0 \xleftrightarrow{\beta_{0,1}} X_1 \xleftrightarrow{\beta_{1,2}} \cdots \xleftrightarrow{\beta_{n-1,n}} X_n).
\end{equation}
The diagram $ZX$ is equivalent to $\widetilde{ZX}$ `up to homotopy,' in the sense that only homotopy equivalences have been removed.

\section{The space of sections of a tame function} \label{sec:sectionsoftamefunction}

Let $X$ be a compact cobordism of manifolds with boundary (Definition~\ref{dfn:cobordism}) and $f : X \rightarrow [s_-,s_+]$ a tame function (Definition~\ref{dfn:tamefunction}).

\begin{dfn} \label{dfn:sectionoftamefunction}
A {\bf section} of $f$ is a continuous map $\delta : [s_-,s_+] \rightarrow X$ satisfying $f \circ \delta = id_{[s_-,s_+]}$, i.e.,\ $\delta(t) \in f^{-1}(t)$ for all $t \in [s_-,s_+]$.
The {\bf space of sections} of $f$ is the space of continuous sections
\begin{equation*} \label{eq:sectionsofIspace}
\G f := \{ \delta : [s_-,s_+] \rightarrow X \,\, | \,\, f \circ \delta = id_{[s_-,s_+]} \}
\end{equation*}
with the compact-open topology.
\end{dfn}

We are interested in computing the connected components $\pi_0(\G f)$ in terms of the homotopy groups of pre-images $f^{-1}(U)$ of open sets $U \subset [s_-,s_+]$. In practice, we perform constructions with regular fibers $f^{-1}(t)$ and their inclusions into cobordisms between them. The formalism we use is the zigzag discretization described in \textsection \ref{subsec:zigzagdiscretization}, which is equivalent to working with preimages of open sets by Lemma~\ref{lem:trivialcobordismnearregularvalue}.

The main results of this section are the computation of $\pi_0(\G f)$ in Theorem~\ref{thm:computationconnectedcomponents} and a similar Theorem~\ref{thm:computationcomponentsS1} when $X$ is a compact manifold with boundary and $f : X \rightarrow S^1$ is submersive with isolated critical points on the boundary with distinct critical values. The proof is the series of arguments in \textsection\ref{subsec:connectedcomponents}. It relies on the section splicing and collapse constructions developed in \textsection\ref{subsec:splicingcollapse}.

\subsection{Section splicing and collapse} \label{subsec:splicingcollapse}
We establish several constructions on homotopy classes of sections and their properties.

Consider a section $\delta : [s_-,s_+] \rightarrow X$ of $f$ and a path $\gamma : [0,1] \rightarrow X_-$ satisfying $\gamma(1) = \delta(s_-)$. The idea is to splice $\gamma$ into $\delta$ in a trivial cobordism near $s_-$. Precisely, choose $\epsilon > 0$ small enough so that the interval $[s_-, s_-+\epsilon]$ does not contain any critical values. Then, by Lemma~\ref{lem:trivialcobordismnearregularvalue}, there is a fiberwise diffeomorphism
$$\varphi : f^{-1}([s_-, s_- +\epsilon]) \cong X_- \times [s_-, s_- +\epsilon]$$
that restricts to the identity $X_- \rightarrow X_- \times \{s_-\}.$ Here fiberwise means that
$$f =  \pi_{[s_-, s_- +\epsilon]} \circ \varphi$$
on $f^{-1}([s_-, s_- +\epsilon]) \subset X$, where $\pi_{[s_-, s_- +\epsilon]} : X_- \times [s_-, s_- +\epsilon] \rightarrow [s_-,s_-+\epsilon]$ is the projection. There is also a projection $\pi : X_- \times [s_-, s_- +\epsilon] \rightarrow X_-$.

\begin{dfn} \label{dfn:leftboundarysplicingepsilonformula}
Given a section $\delta : [s_-,s_+] \rightarrow X$ of $f$ and a path $\gamma : [0,1] \rightarrow X_-$ satisfying $\gamma(1) = \delta(s_-)$,
the {\bf left boundary $\epsilon$-splicing}
$$(\gamma \,\,\#_-^{\epsilon}\,\, \delta) : [s_-,s_+] \rightarrow X$$
is the section of $f$ defined by
\begin{equation*}
(\gamma \,\,\#_-^{\epsilon}\,\, \delta)(t) = 
\begin{cases} 
	\varphi^{-1}(\gamma(\frac{2}{\epsilon}(t-s_-)),t) &\text{ if } t \in [s_-, s_-+\epsilon/2]\\
        \varphi^{-1}(\pi(\varphi(\delta(2t - s_- - \epsilon))),t) &\text{ if } t \in [s_-+\epsilon/2,s_-+\epsilon]\\
        \delta(t) &\text{ if } t \in [s_-+\epsilon,s_+]. \\
\end{cases}
\end{equation*}

For a path $\gamma : [0,1] \rightarrow X_+$ satisfying $\gamma(0) = \delta(s_+)$, there is a {\bf right boundary $\epsilon$-splicing} $\delta \,\,\#_+^{\epsilon}\,\, \gamma$ defined by gluing $\gamma$ into a regular cobordism near $s_+$ using the symmetric formula.

Given a regular value $s_* \in (s_-,s_+)$ and a path $\gamma : [0,1] \rightarrow X_{s_*} = f^{-1}(s_*)$ satisfying $\gamma(1) = \delta(s_*)$, the {\bf interior $\epsilon$-splicing} of the path $\gamma$ into $\delta$ at $s_*$ is the section of $f$ given by
\begin{align*}
\cI^{\epsilon}(\delta,\g,s_*) &:= (\delta|_{[s_-,s_*]} \,\,\#_+\,\, \overline{\g}, \,\, \g \,\,\#_-\,\, \delta|_{[s_*,s_+]})\\
&=
\begin{cases} 
	(\delta|_{[s_-,s_*]} \,\,\#_+\,\, \overline{\g})(t) &\text{ if } t \in [s_-,s_*]\\
	(\g \,\,\#_-\,\, \delta|_{[s_*,s_+]})(t) &\text{ if } t \in [s_*,s_+],
\end{cases}
\end{align*}
where
$$\overline{\g}(t) := \g(1-t)$$
is the reverse path, $\delta|_{[s_-,s_*]} : [s_-,s_*] \rightarrow f^{-1}([s_-,s_*])$ is viewed as a section of $f^{-1}([s_-,s_*]) \rightarrow [s_-,s_*]$, and similarly for $\delta|_{[s_*,s_+]}$.
\end{dfn}

The left boundary $\epsilon$-splicing $(\gamma \,\,\#_-^{\epsilon}\,\, \delta) : [s_-,s_+] \rightarrow X$
is a section of $f$ with endpoints
$$(\gamma \,\,\#_-^{\epsilon}\,\, \delta)(s_-) = \g(0) \text{ and } (\gamma \,\,\#_-^{\epsilon}\,\, \delta)(s_+) = \delta(s_+).$$
Similarly, the right boundary splicing has endpoints
$$(\delta \,\,\#_+^{\epsilon}\,\, \gamma)(s_-) = \delta(s_-)\text{ and } (\delta \,\,\#_+^{\epsilon}\,\, \gamma)(s_+) = \gamma(1).$$

\begin{lem} \label{lem:constantmapsplicing}
Consider a section $\delta \in \G f$ of $f : X \rightarrow [s_-,s_+]$ and the constant paths $\g_- \equiv \delta(s_-)$ and $\g_+ \equiv \delta(s_+)$. Then the splicings $\g_- \,\, \#_-^{\epsilon} \,\, \delta$ and $\delta \,\, \#_+^{\epsilon} \g_+$ are fiberwise homotopic to $\delta$.
\end{lem}
\begin{proof}
A fiberwise homotopy $h : [0,1] \times [s_-,s_+] \rightarrow X$ from $\delta$ to $\g_- \,\, \#_-^{\epsilon} \,\, \delta$ is given by
\begin{align*}
h(r,t) :=
\begin{cases} 
	\varphi^{-1}(\delta(s_-),t) &\text{ if } t \in [s_-, s_-+r\epsilon/2]\\
        \varphi^{-1}(\pi(\varphi(\delta(2t - s_- - r\epsilon))),t) &\text{ if } t \in [s_-+r\epsilon/2,s_-+r\epsilon]\\
        \delta(t) &\text{ if } t \in [s_-+r\epsilon,s_+]. \\
\end{cases}
\end{align*}
\end{proof}

As a first application of splicings, we use them in the proof of the following lemma to modify a section to have specific behavior near the endpoints.

\begin{lem} \label{lem:singlecriticalpointlift}
Let $f : X \rightarrow [s_-,s_+]$ be a tame function. Assume that either all critical points of $f|_B$ are type D, or all critical points are type N. Let $\delta_- \in X_-$ and $\delta_+ \in X_+$ such that $\delta_-$ and $\delta_+$ are in the same connected component of $X$. Then there exists a continuous section $\delta : [s_-,s_+] \rightarrow X$ of $f$ such that $\delta(s_-) = \delta_-$ and $\delta(s_+) = \delta_+$.
\end{lem}
\begin{proof}
Assume that all critical points of $f|_B$ are type D; the argument for type N critical points is symmetric.

Let $H: X \times [0,s_+ - s_-] \rightarrow X$ be the deformation retraction from Lemma~\ref{lem:defretracttoregvalues} satisfying $$f(H(x,t)) = \max\{s_-, \,\, f(x) - t\}.$$
Define the section
\begin{align*}
\tilde{\delta} : [s_-,s_+] &\rightarrow X\\
t &\mapsto H(\delta_+,s_+-t).
\end{align*}
of $f$. It satisfies $\tilde{\delta}(s_+) = H(\delta_+,0) = \delta_+$ and $\tilde{\delta}(s_-) \in X_- = f^{-1}(s_-)$.

We now modify $\tilde{\delta}$ near $s_-$ using a left boundary splicing to obtain a section $\delta : [s_-,s_+] \rightarrow X$ that also satisfies $\delta(s_-) = \delta_-$. Indeed, since $\tilde{\delta}(t)$ is in the same component of $X$ as $\delta_+ = \tilde{\delta}(s_+)$ for all $t \in [s_-,s_+]$, which is in the same component of $X$ as $\delta_-$ by hypothesis, it follows that $\tilde{\delta}(s_-)$ is in the same component of $X$ as $\delta_-$. Taking a path in $X$ from  $\delta_-$ to $\tilde{\delta}(s_-)$ and deformation retracting it with $H$ into $X_-$ produces a path $\gamma : [0,1] \rightarrow X_-$ such that $\g(0) = \delta_-$ and $\g(1) = \tilde{\delta}(s_-)$. Then the left boundary $\epsilon$-splicing (Definition~\ref{dfn:leftboundarysplicingepsilonformula})
$$\delta := (\g \,\,\#_-^{\epsilon}\,\, \tilde{\delta}) : [s_-,s_+] \rightarrow X$$
for any sufficiently small $\epsilon > 0$ is a section of $f$ satisfying $\delta(s_-) = \g(0) = \delta_-$ and $\delta(s_+) = \tilde{\delta}(s_+) = \delta_+$, as required.
\end{proof}

Splicings are also useful for converting fiberwise homotopies of sections with free endpoints to fiberwise homotopies with fixed endpoints.

\begin{lem} \label{lem:endpointfreetoendpointfixed}
Let $f : X \rightarrow [s_-,s_+]$ be a tame function with sections $\delta, \hat{\delta} \in \G f$. Suppose there is a fiberwise homotopy $h: [0,1] \times [s_-,s_+] \rightarrow X$ from $\hat{\delta}$ to $\delta$, meaning that $f(h(r,t)) = t$, $\hat{\delta}(t) = h(0,t)$, and $\delta(t) = h(1,t)$ for all $r,t$. Consider the paths $\g_{\pm} := h(-,s_{\pm}) : [0,1] \rightarrow X_{\pm}$. Then there is a fiberwise homotopy $\tilde{h} : [0,1] \times [s_-,s_+] \rightarrow X$ from $\g_- \,\,\#_-^{\epsilon}\,\, \delta \,\,\#_+^{\epsilon}\,\, \overline{\g}_+$ to $\hat{\delta}$ that is endpoint preserving, i.e.\ $\tilde{h}(r,s_{\pm}) = \hat{\delta}(s_{\pm})$ for all $r \in [0,1]$.
\end{lem}
\begin{proof}
For $r \in [0,1]$, consider the paths $\g_{\pm}^r(t) := \g_{\pm}(rt).$ Let $c_{\pm}(t) := \hat{\delta}(s_{\pm})$ be the constant paths at the endpoints of $\hat{\delta}$. Then
$\tilde{h}(r,\cdot) := \gamma^r_- \,\,\#_-^{\epsilon}\,\, h(r,\cdot) \,\,\#_+^{\epsilon}\,\, \overline{\gamma^r_+}$
is a fiberwise homotopy from $\tilde{h}(0,\cdot) = c_- \,\,\#_-^{\epsilon}\,\, \hat{\delta} \,\,\#_+^{\epsilon}\,\, c_+$ to $\tilde{h}(1,\cdot) = \gamma_- \,\,\#_-^{\epsilon}\,\, \delta \,\,\#_+^{\epsilon}\,\, \overline{\gamma}_+$ and is endpoint preserving. Finally, the section $c_- \,\,\#_-^{\epsilon}\,\, \hat{\delta} \,\,\#_+^{\epsilon}\,\, c_+$ is fiberwise homotopic to $\hat{\delta}$ by Lemma~\ref{lem:constantmapsplicing}, and the homotopy preserves endpoints.
\end{proof}

The next lemma shows that boundary splicing and interior splicing do not affect the fiberwise homotopy class of a section with no endpoint constraints.

\begin{lem} \label{lem:freehomotopytypeindependentsplicing}
Let $f : X \rightarrow [s_-,s_+]$ be a tame function with a section $\delta \in \G f$, paths $\g_{\pm} : [0,1] \rightarrow X_{\pm}$ such that $\g_-(1) = \delta(s_-)$ and $\g_+(0) = \delta(s_+)$, and a path $\g : [0,1] \rightarrow f^{-1}(s_*)$ for some regular value $s_*$ of $f|_B$ such that $\g(1) = \delta(s_*)$. Then the left boundary splicing $\g_- \,\,\#_-^{\epsilon}\,\, \delta$, the right boundary splicing $\delta \,\,\#_+^{\epsilon}\,\, \g_+$, and the interior splicing $\cI^{\epsilon}(\delta,\g,s_*)$ are all fiberwise homotopic with free endpoints to $\delta$. That is, all of these sections are in the same path component of $\G f$ as $\delta$.
\end{lem}
\begin{proof}
Since the endpoints are allowed to move on $X_{\pm} = f^{-1}(s_{\pm})$ during the homotopy, we can `unwind' the path $\g_{\pm}$ spliced onto the left or the right in a trivial cobordism near $s_{\pm}$. The interior splicing $\cI^{\epsilon}(\delta,\g,s_*)$ inserts the graph of the path $\g$ next to the graph of $\overline{\g}$ in a trivial cobordism near $f^{-1}(s_*)$. The concatenation of these paths in $f^{-1}(s_*)$ is a loop based at $\delta(s_*)$ and is contractible relative the basepoint. The graph of this contraction provides a fiberwise homotopy from the interior splicing to $\delta$.
\end{proof}

We now introduce some notation. Given points $x_- \in X_-$ and $x_+ \in X_+$, define the space of sections with fixed endpoints
$$\G f_{x_-}^{x^+} := \{ \delta \in \G f \,\, | \,\, \delta(s_-) = x_- \text{ and } \delta(s_+) = x_+\}.$$
Given a space $A$ and points $x,x' \in A$, the path space with fixed endpoints is denoted
$$\cP(A)^x_{x'} := \{ \g : [0,1] \rightarrow A \,\, | \,\, \g(0) = x' \text{ and } \g(1) = x \}.$$
The based loop space of $A$ at $x$ is denoted
$$\Omega(A)_x = \cP(A)_x^x.$$

The fiberwise homotopy class of a left boundary splicing $\gamma \,\,\#_-^{\epsilon}\,\, \delta$ relative its endpoints $[\gamma \,\,\#_-^{\epsilon}\,\, \delta] \in \pi_0(\G f_{\g(0)}^{\delta(s_+)})$ does not depend on the homotopy class of $\g$ relative its endpoints $[\g] \in \pi_0(\cP(X_-)_{\g(0)}^{\g(1)})$, the fiberwise homotopy class of $\delta$ relative its endpoints $[\delta] \in \pi_0(\G f_{\delta(s_-)}^{\delta(s_+)})$, or $\epsilon > 0$. In particular, for any $x_-', x_- \in X_-$ and $x_+ \in X_+$, left boundary $\epsilon$-splicing descends to a well-defined map
\begin{align*}
\#_- : \pi_0(\cP(X_-)_{x_-'}^{x_-}) \times \pi_0(\G f_{x_-}^{x_+}) &\rightarrow \pi_0(\G f_{x_-'}^{x_+})\\
([\gamma], [\delta]) &\mapsto [\gamma \,\,\#_-^{\epsilon}\,\, \delta].
\end{align*}
In the case $x'_- = x_-$ the path space $\cP(X_-)_{x_-'}^{x_-}$ is equal to the based loop space $\Omega (X_-)_{x_-}$ and moreover $\pi_0(\cP(X_-)_{x_-'}^{x_-}) = \pi_1(X_-,x_-)$. The map
\begin{align*}
\#_- : \pi_1(X_-,x_-) \times \pi_0(\G f_{x_-}^{x_+}) &\rightarrow \pi_0(\G f_{x_-}^{x_+})
\end{align*}
is a left group action of $\pi_1(X_-,x_-)$ on the set $\pi_0(\G f_{x_-}^{x_+}).$

Similarly, given $x_- \in X_-$ and $x_+,x_+' \in X_+$, right boundary $\epsilon$-splicing descends to a well-defined map
\begin{align*}
\#_+ :  \pi_0(\G f_{x_-}^{x_+}) \times \pi_0(\cP(X_+)_{x_+}^{x_+'}) &\rightarrow \pi_0(\G f_{x_-}^{x_+'})\\
([\delta], [\gamma]) &\mapsto [\delta \,\,\#_+^{\epsilon}\,\, \gamma].
\end{align*}
In the case $x'_+ = x_+$, the map
\begin{align*}
\#_+ :  \pi_0(\G f_{x_-}^{x_+}) \times \pi_1(X_+,x_+) &\rightarrow \pi_0(\G f_{x_-}^{x_+})
\end{align*}
is a right group action of $\pi_1(X_+,x_+)$ on the set $\pi_0(\G f_{x_-}^{x_+}).$

We now define various versions of a {\bf collapse map} and establish their properties. Let $f : X \rightarrow [s_-, s_+]$ be a tame function with only type $D$ critical points or, symmetrically, with only type $N$ critical points. For the rest of this section, the symbol $\pm$ means $-$ if type $D$ and $+$ if type $N$, and symmetrically for the symbol $\mp$.

Fix a basepoint section $\mathfrak{b} \in \G f$. Let $H : X \times [0, s_+ - s_-] \rightarrow X$ be a deformation retraction along $\mathfrak{b}$ and onto $X_{\pm}$, as provided by Lemma~\ref{lem:defretracttoregvalues}. Then define the {\bf cobordism collapse map}
\begin{equation}  \label{eq:collapsemap}
C := H(-,s_+-s_-) : X \rightarrow X_{\pm}.
\end{equation}
The map $C$ is a left inverse of the inclusion $X_{\pm} \rightarrow X$. Since $H(\mathfrak{b}(t),s_+-s_-) = \mathfrak{b}(s_{\pm})$ for all $t \in [s_-,s_+]$ by Lemma~\ref{lem:defretracttoregvalues}, we have
\begin{equation*}
C(\mathfrak{b}(t)) = \mathfrak{b}(s_{\pm})\text{ for all } t \in [s_-,s_+],
\end{equation*}
and so $C$ is a map of pairs
\begin{equation*}
C : (X,\mathfrak{b}([s_-,s_+])) \rightarrow (X_{\pm},\mathfrak{b}(s_{\pm})).
\end{equation*}
This map induces the {\bf section collapse map} on components
\begin{align} \label{eq:sectioncollapsemap}
\G C_* : \pi_0(\G f_{\mathfrak{b}(s_-)}^{\mathfrak{b}(s_+)}) &\rightarrow  \pi_1(X_{\pm},\mathfrak{b}(s_{\pm}))\\
[\delta] &\mapsto [C \circ \delta]. \nonumber
\end{align}
We prove in Lemma~\ref{lem:sectioncollapsebijectivity} that $\G C_*$ is a bijection. There is a similarly defined {\bf loop collapse map}
\begin{align*}
C_* : \pi_1(X_{\mp},\mathfrak{b}(s_{\mp})) &\rightarrow  \pi_1(X_{\pm},\mathfrak{b}(s_{\pm}))\\
[\g] &\mapsto [C \circ \g]. \nonumber
\end{align*}
Section collapse $\G C_*$ and section splicing $\#_{\pm}$ are related as follows.

\begin{lem} \label{lem:relationshipsplicingcollapse}
Let $f : X \rightarrow [s_-,s_+]$ be a tame function with only type $D$ critical points and $\mathfrak{b} \in \G f$. Let $C : X \rightarrow X_-$ be an associated collapse map along $\mathfrak{b}$.

Consider a section $\delta \in \G f_{\mathfrak{b}(s_-)}^{\mathfrak{b}(s_+)}$. Then, in $\pi_0(\G f_{\mathfrak{b}(s_-)}^{\mathfrak{b}(s_+)})$, we have
$$\G C_*([\delta]) \,\,\#_-\,\, [\mathfrak{b}] = [\delta].$$

Moreover, consider loops $\g_- \in \pi_1(X_-,\mathfrak{b}(s_-))$ and $\g_+ \in \pi_1(X_+,\mathfrak{b}(s_+))$. Then, in $\pi_1(X_-,\mathfrak{b}(s_-))$ with group operation denoted $\cdot$,  it holds that
$$\G C_*(\g_- \,\,\#_-\,\, \delta \,\,\#_+\,\, \g_+) = \g_- \cdot \G C_*(\delta) \cdot C_*(\g_+).$$
In particular,
$$\G C_*(\g_- \,\,\#_-\,\, \delta) = \g_- \cdot \G C_*(\delta)$$
and
$$\G C_*(\delta \,\,\#_+\,\, \g_+) = \G C_*(\delta) \cdot C_*(\g_+).$$
The symmetric statement holds if $f$ has only type $N$ critical points.
\end{lem}
\begin{proof}
To prove $\G C_*([\delta]) \,\,\#_-\,\, [\mathfrak{b}] = [\delta],$ we must show that for small $\epsilon > 0$ there is a fiberwise homotopy from $\delta$ to the section $(C \circ \delta) \,\,\#_-^{\epsilon}\,\, \mathfrak{b}$. The section $(C \circ \delta) \,\,\#_-^{\epsilon}\,\, \mathfrak{b}$ is given by flowing $\delta$ to the left until it lies in a trivial cobordism near $X_-$ with its right endpoint lying on $\mathfrak{b}$ and then attaching this onto the left of $\mathfrak{b}$. The homotopy is constructed using this flow. We now write this out in precise formulas.

For simplification of the formulas, assume without loss of generality that $s_- = 0$ and $s_+ = 1$. By definition, we have
$$(C \circ \delta)(t) = H(\delta(t), 1) \text{ for } t \in [0,1]$$
and
\begin{equation*}
((C \circ \delta) \,\,\#_-^{\epsilon}\,\, \mathfrak{b})(t) = 
\begin{cases} 
	\varphi^{-1}(H(\delta(\frac{2}{\epsilon}t),1),t) &\text{ if } t \in [0, \epsilon/2]\\
        \varphi^{-1}(\pi(\varphi(\mathfrak{b}(2t - \epsilon))),t) &\text{ if } t \in [\epsilon/2,\epsilon]\\
        \mathfrak{b}(t) &\text{ if } t \in [\epsilon,1].
\end{cases}
\end{equation*}
The claimed homotopy is then given by
\begin{equation*}
(r,t) \mapsto 
\begin{cases} 
	\varphi^{-1}(\pi(\varphi(H(\delta(\frac{2t}{r\epsilon}(t+r(1-t))),r))),t) & \text{ if } t \in [0, \frac{r\epsilon}{2}]\\
        \varphi^{-1}(\pi(\varphi(H(\delta(t + r(1-t)),r(1 - (2t - r\epsilon))))),t) & \text{ if } t \in [\frac{r\epsilon}{2},r\epsilon]\\
        H(\delta(t + r(1 - t)), r(1-t)) & \text{ if } t \in [r\epsilon,1].
\end{cases}
\end{equation*}
for $r \in [0,1]$. Note that the $r$ in the denominator in the case $t \in [0, \frac{r\epsilon}{2}]$ is not a continuity issue because we have
$$0 \leq \frac{2t}{r\epsilon}(t+r(1-t)) = \frac{2t}{r\epsilon}(t(1-r) + r)  \leq \frac{r\epsilon}{2}(1-r) + r$$ and hence the limit as $r \rightarrow 0$ is equal to zero. This completes the proof of the claim $\G C_*([\delta]) \,\,\#_-\,\, [\mathfrak{b}] = [\delta]$.

The next claim $\G C_*(\g_- \,\,\#_-\,\, \delta \,\,\#_+\,\, \g_+) = \g_- \cdot \G C_*(\delta) \cdot C_*(\g_+)$ is seen as follows. The left side is obtained by flowing $\g_- \,\,\#_-\,\, \delta \,\,\#_+\,\, \g_+$ to the left until it lies in $X_-$. The result of this is the loop $\g_-$, followed by the section $\delta$ pushed into $X_-$, which is exactly $\G C_*(\delta)$, followed by the loop $\g_+$ pushed from $X_+$ into $X_-$ via the flow, which is $C_*(\g_+)$. Appending one path after another is the composition $\cdot$ in $\pi_1(X_-,\mathfrak{b}(s_-))$ on the right hand side of the claimed formula.
\end{proof}

An important consequence of Lemma~\ref{lem:relationshipsplicingcollapse} is that the section collapse map $\G C_*$ is bijective.

\begin{lem} \label{lem:sectioncollapsebijectivity}
Let $f : X \rightarrow [s_-,s_+]$ be a tame function with only type $D$ critical points or only type $N$ critical points, and let $\mathfrak{b} \in \G f$. Then the section collapse map $\G C_*$ defined in \eqref{eq:sectioncollapsemap} is bijective.
\end{lem}
\begin{proof}
Assume all critical points are type $D$. For surjectivity, let $\g \in \pi_1(X_-,\mathfrak{b}(s_-))$ and consider the left boundary splicing $\delta := \g \,\,\#_-\,\, \mathfrak{b}.$ Then by Lemma~\ref{lem:relationshipsplicingcollapse} we have $\G C_*(\delta) = \g \cdot \G C_*(\mathfrak{b}) = \g \cdot 1 = \g$. For injectivity, let $\delta,\hat{\delta} \in \pi_0(\G f_{\mathfrak{b}(s_-)}^{\mathfrak{b}(s_+)})$ and assume $\G C_*(\delta) = \G C_*(\hat{\delta}).$ Then by Lemma~\ref{lem:relationshipsplicingcollapse} we have $\delta = \G C_*(\delta) \,\,\#_-\,\, [\mathfrak{b}] = \G C_*(\hat{\delta}) \,\,\#_-\,\, [\mathfrak{b}] = \hat{\delta}.$
\end{proof}

\subsection{Connected components} \label{subsec:connectedcomponents}

The connected components $\pi_0(\G f)$ of the space of sections $\G f$ of a tame function $f : X \rightarrow I = [0,1]$ can be computed from $\pi_0$ and $\pi_1$ of the regular fibers of $f$ and maps between them coming from the regular cobordisms between them. This is the main theorem (Theorem~\ref{thm:computationconnectedcomponents}) proved in this section. There is a similar statement for maps $X \rightarrow S^1$ given in Theorem~\ref{thm:computationcomponentsS1}.

Recall the diagram $ZX$ of regular fibers of $f$ from \eqref{eq:zigzagX}. Applying $\pi_0$ produces a diagram of sets
$$\pi_0(ZX) = \bigg ( \pi_0(X_0) \leftrightarrow \pi_0(X_1) \leftrightarrow \cdots \leftrightarrow \pi_0(X_n) \bigg ).$$

A central object in this analysis is the inverse limit $\varprojlim \pi_0(ZX)$. Recall that an element $\Psi \in \varprojlim \pi_0(ZX)$ is a collection
\begin{align*}
\Psi_{s_i} &\in \pi_0(X_i) \text{ for } 0 \leq i \leq n,
\end{align*}
such that, for all $i$, it holds that ${\beta_{i,i+1}}_*(\Psi_{s_{\partial^-(i)}}) = \Psi_{s_{\partial^+(i)}}$ where ${\beta_{i,i+1}}_* : \pi_0(X_{\partial^-(i)}) \rightarrow \pi_0(X_{\partial^+(i)})$ is the map on $\pi_0$ induced by the map $\beta_{i,i+1}$ defined in \eqref{eq:betadirected}.

The connected components $\pi_0(\G f)$  of the space of sections $\G f$ are related to the inverse limit $\varprojlim \pi_0(ZX)$ via the map
\begin{align*}
\Pi_0 : \pi_0(\G f) &\rightarrow \varprojlim \pi_0(ZX)\\
[\delta] &\mapsto
\{ \Psi_{s_i} := [\delta(s_i)] \in \pi_0(X_i) \text{ for }  i = 0,\ldots,n \}. \nonumber
\end{align*}

Note that $\pi_0(\G f)$ is the set of fiber-preserving homotopy classes of sections of $f$.

The map $\Pi_0$ is surjective, as we prove in Proposition~\ref{prp:surjectivity}. The idea of the proof is as follows. We must lift a given $\Psi \in \varprojlim \pi_0(ZX)$ through $\Pi_0$ to a section $I \rightarrow X$ of $f$. For each cobordism $X_i^{i+1} = f^{-1}([s_i,s_{i+1}])$ of manifolds with boundary $X_i = f^{-1}(s_i)$ and $X_{i+1} = f^{-1}(s_{i+1})$, Lemma~\ref{lem:singlecriticalpointlift} provides a lift over the interval $[s_i,s_{i+1}]$. These lifts agree at the regular values $s_i$, hence they fit together to form the desired lift of $\Psi$.

\begin{prp} \label{prp:surjectivity} 
The map $\Pi_0 : \pi_0(\G f) \rightarrow \varprojlim \pi_0(ZX)$ is surjective.
\end{prp}
\begin{proof}
Let $\Psi \in \varprojlim \pi_0(ZX)$ given by $\Psi_{s_i} \in \pi_0(X_i) \text{ for } 0 \leq i \leq n$. To prove the lemma, we must lift $\Psi$ through $\Pi_0$ to a continuous section $\delta : [0,1] \rightarrow X$ satisfying
\begin{equation} \label{eq:liftofPi}
[\delta(s_i)] = \Psi_{s_i}
\end{equation}
for all $0 \leq i \leq n$.

For each $0 \leq i \leq n$, choose a lift of $\Psi_{s_i}$, i.e.\ a point
$$\delta_i \in \Psi_{s_i}.$$
For each $0 \leq i \leq n-1$, the restriction
$$f|_{X_i^{i+1}} : X_i^{i+1} \rightarrow [s_i,s_{i+1}]$$
is a tame function on the compact cobordism $X_i^{i+1}$ between the manifolds with boundary $X_i$ and $X_{i+1}$.
We claim that the hypotheses of Lemma~\ref{lem:singlecriticalpointlift} hold, providing a continuous section
$$\delta_{i,i+1} : [s_i,s_{i+1}] \rightarrow X_i^{i+1}$$
such that
$$\delta_{i,i+1}(s_i) = \delta_i \text{ and } \delta_{i,i+1}(s_{i+1}) = \delta_{s_{i+1}}.$$
Indeed, since $[s_i,s_{i+1}]$ contains exactly $1$ critical value of $f|_B$ by the choice of interleaving regular values $s_i$, the function $f|_{X_i^{i+1}}$ has exactly $1$ critical point on the boundary. Moreover, $\delta_i$ and $\delta_{i+1}$ are contained in the same connected component of $X_i^{i+1}$ since ${\beta_{i,i+1}}_*(\Psi_{s_{\partial^-(i)}}) = \Psi_{s_{\partial^+(i)}}$.

The endpoints of the sections $\delta_{i,i+1}$ agree, i.e., $\delta_{i,i+1}(s_{i+1}) = \delta_{i+1,i+2}(s_{i+1})$ for all $0 \leq i \leq n-2$. Hence they fit together to form a continuous section $\delta : I \rightarrow X$ of $f$ which satisfies \eqref{eq:liftofPi}. Hence $\Pi_0([\delta]) = \Psi$ and the proof is complete.
\end{proof}

To calculate $\pi_0(\G f)$, it remains to characterize the fibers of the surjection $\Pi_0 : \pi_0(\G f) \rightarrow \varprojlim \pi_0(Z X)$.

Let $\Psi \in \varprojlim \pi_0(ZX)$. Choose a fixed `basepoint' section $\mathfrak{b} \in \G f$ such that
$$\Pi_0([\mathfrak{b}]) = \Psi,$$
which means that
$$\mathfrak{b}(s_i) \in \Psi_{s_i} \text{ for } 0 \leq i \leq n.$$
As an intermediate object in our analysis, we consider the subspace of sections that agree with $\mathfrak{b}$ at regular values $s_i$; precisely, we define
$$\G f (\mathfrak{b}) := \{ \delta \in \G f \,\, | \,\, \delta(s_i) = \mathfrak{b}(s_i) \text{ for all } i \}.$$
This space splits as the Cartesian product of the spaces of sections of the restrictions $f|_{X_i^{i+1}} : X_i^{i+1} \rightarrow [s_i,s_{i+1}]$ with endpoints agreeing with $\mathfrak{b}$, i.e.\
$$\G f (\mathfrak{b}) = \prod_{i=0}^{n-1} \G f|_{X_i^{i+1}} ( \mathfrak{b}|_{[s_i,s_{i+1}]}).$$
Notice that, for every $\delta \in \G f (\mathfrak{b})$, we have $\Pi_0([\delta]) = \Psi$. It follows that the inclusion $\i : \G f (\mathfrak{b})  \hookrightarrow \G f$ induces a map
$$\pi_0(\G f (\mathfrak{b})) \xrightarrow{\i_*} \Pi_0^{-1}(\Psi) \subset \pi_0(\G f).$$

\begin{prp} \label{prp:surjectivityontofiber}
The map $\pi_0(\G f (\mathfrak{b})) \xrightarrow{\i_*} \Pi_0^{-1}(\Psi)$ is surjective.
\end{prp}
\begin{proof}
Let $\delta \in \G f$ such that $[\delta] \in \Pi_0^{-1}(\Psi)$. This means that $\delta(s_i) \in \Psi_{s_i}$ for all $i = 0,\ldots,n$. Since also $\mathfrak{b}(s_i) \in \Psi_{s_i}$ for all $i$, there is a continuous path $\gamma_i : [0,1] \rightarrow \Psi_{s_i}$ from $\g_i(0) = \mathfrak{b}(s_i)$ to $\g_i(1) = \delta(s_i)$. The idea is to splice $\g_i$ into $\delta$ on the left of $s_i$ and to splice $\overline{\g}_i(\cdot) := \g_i(1-\cdot)$ into $\delta$ on the right of $s_i$, producing a new section $\hat{\delta} \in \G f (\mathfrak{b})$, i.e.\ $\hat{\delta}(s_i) = \mathfrak{b}(s_i)$ for all $i$. Moreover, in the space of sections $\G f$ with no point restrictions, this spliced section $\hat{\delta}$ is in the same path component as $\delta$, i.e.\ $\hat{\delta}$ is fiberwise homotopic to $\delta$ by Lemma~\ref{lem:freehomotopytypeindependentsplicing}. Then $\i_*([\hat{\delta}]) = [\delta]$ and the proof is complete.

We precisely construct the spliced section $\hat{\delta}$ and verify the claimed properties. Restricting $\delta$ to subintervals produces a section $\delta_i : [s_i,s_{i+1}] \rightarrow X_i^{i+1}$ of the tame function $f|_{X_i^{i+1}} \rightarrow [s_i,s_{i+1}]$ for $i = 0,\ldots,n-1$. Form the sections
$$\hat{\delta}_i := \big ( \g_i \,\,\#_-^{\epsilon}\,\, \delta_i \,\,\#_+^{\epsilon}\,\, \overline{\g_{i+1}} \big ) : [s_i,s_{i+1}] \rightarrow X_i^{i+1}$$
for some $\epsilon > 0$ small enough. Then for all $0 \leq i \leq n-1$ it holds that $\hat{\delta}_i(s_i) = \mathfrak{b}(s_i)$ and $\hat{\delta}_i(s_{i+1}) = \mathfrak{b}(s_{i+1})$. In particular, $\hat{\delta}_i(s_{i+1}) = \hat{\delta}_{i+1}(s_{i+1})$, and hence the $\hat{\delta}_i$ fit together to form a section $\hat{\delta} : [0,1] \rightarrow X$ of $f$. Moreover, $\hat{\delta} \in \G f (\mathfrak{b})$ since $\hat{\delta}(s_i) = \mathfrak{b}(s_i)$ for all $i$.

It remains to show that $\hat{\delta}$ is fiberwise homotopic to $\delta$ with no point restrictions. Observe that $\hat{\delta}$ is formed by starting with $\delta$, performing a left boundary splicing with $\g_0$, then performing an interior splicing with $\g_i$ for all $1 \leq i \leq n-1$, and finally performing a right boundary splicing with $\gamma_n$. Hence the claim follows from repeated application of Lemma~\ref{lem:freehomotopytypeindependentsplicing}.
\end{proof}

Consider the direct product group $\prod_{i=0}^n \pi_1(X_i,\mathfrak{b}(s_i))$. There is a group action on the set $\pi_0(\G f (\mathfrak{b}) )$, denoted by $\star$ and defined in \eqref{eq:actiononbasedsections},
$$\star : \prod_{i=0}^n \pi_1(X_i,\mathfrak{b}(s_i)) \times \pi_0(\G f (\mathfrak{b}) ) \rightarrow \pi_0(\G f (\mathfrak{b}) ).$$
The orbits of this action are exactly the fibers of $\i_*$, as we prove in Proposition~\ref{prp:orbitsarefibers}. Hence, due to Proposition~\ref{prp:surjectivityontofiber} above, $\Pi_0^{-1}(\Psi)$ is naturally in bijection with the set of orbits.

We proceed to define the claimed group action and establish its properties. The action of $\t = (\t_0,\ldots,\t_n) \in \prod_{i=0}^n \pi_1(X_i,\mathfrak{b}(s_i))$ on a class $[\delta] \in \pi_0(\G f (\mathfrak{b}) )$ is defined by representing $\delta$ by its restrictions $\delta_i := \delta|_{[s_i,s_{i+1}]} \in \G f|_{X_i^{i+1}}(\mathfrak{b}|_{[s_i,s_{i+1}]})$ for $i = 0,\ldots,n-1$ and taking the spliced homotopy class
\begin{equation} \label{eq:actiononbasedsections}
\t \star [\delta] := (\t_0 \#_- [\delta_0] \#_+ \t_1^{-1}, \t_1 \#_- [\delta_1] \#_+ \t_2^{-1}, \ldots, \t_{n-1} \#_- [\delta_{n-1}] \#_+ \t_n^{-1}).
\end{equation}

\begin{prp} \label{prp:orbitsarefibers}
The orbits of the group action $\star$ are the fibers of $\i_*$. Precisely, this means that $\cO \subset \pi_0(\G f(\mathfrak{b}))$ is an orbit of $\star$ if and only if $\cO = \i_*^{-1}([\delta])$ for some $[\delta] \in \Pi_0^{-1}(\Psi) \subset \pi_0(\G f)$.

In particular, $\i_*$ induces a bijection between the set of orbits of $\star$ and $\Pi_0^{-1}(\Psi)$.
\end{prp}
\begin{proof}
Let $\delta, \hat{\delta} \in \G f (\mathfrak{b})$. It must be shown that $\delta$ is fiberwise homotopic to $\hat{\delta}$ if and only if there exists an element $\t \in \prod_{i=0}^n \pi_i(X_i, \mathfrak{b}(s_i))$ satisfying $\t \star [\delta] = [\hat{\delta}]$.

By the definition \eqref{eq:actiononbasedsections} of the action, for any $\t$ and $\delta$ we see that a representative section of $\t \star [\delta]$ is a sequence of a left boundary splicing with $\t_0$, followed by interior splicing of $\t_i$ at the regular values $s_i$ for all $1 \leq i \leq n-1$, followed by a right boundary splicing of $\t_n$. Hence by Lemma~\ref{lem:freehomotopytypeindependentsplicing} it is fiberwise homotopic to $\delta$. So, if it is assumed that $\t \star [\delta] = [\hat{\delta}]$, it follows that $\hat{\delta}$ is fiberwise homotopic to $\delta$. This proves one direction of the if and only if statement.

To prove the other direction, suppose that $\delta$ is fiberwise homotopic to $\hat{\delta}$. For each $i = 0,\ldots,n-1$, this fiberwise homotopy restricts to a homotopy from the section $\delta_i := \delta|_{[s_i,s_{i+1}]} : [s_i,s_{i+1}] \rightarrow X_i^{i+1}$ to the section $\hat{\delta}_i := \hat{\delta}|_{[s_i,s_{i+1}]}$. Then, Lemma~\ref{lem:endpointfreetoendpointfixed} provides loops $\t_i : [0,1] \rightarrow X_i$ based at $\mathfrak{b}(s_i)$ for $i = 0,\ldots,n$ and a fiberwise homotopy from
$\t_i \,\,\#_-^{\epsilon}\,\, \delta_i \,\,\#_+^{\epsilon}\,\, {\t}^{-1}_{i+1}$ to $\hat{\delta}_i$ that preserves the endpoints at $\mathfrak{b}(s_i)$ and $\mathfrak{b}(s_{i+1})$. Since these homotopies are endpoint preserving, setting $\t := (\t_0,\ldots,\t_n)$ they fit together to provide a homotopy from a representative of $\t \star [\delta]$ to  $\hat{\delta}$ within the space $\G f(\mathfrak{b})$. Hence $\t \star [\delta] = [\hat{\delta}]$, as required.
\end{proof}

In light of Proposition~\ref{prp:orbitsarefibers}, to calculate $\Pi_0^{-1}(\Psi)$ it suffices to understand the group action $\star$. The group acting is defined in the regular fibers of $f$ -- it is a product of fundamental groups of fibers -- however the set $\pi_0(\G f (\mathfrak{b}))$ and the action have not yet been algebraically described in terms of homotopy theoretic information about the fibers. We proceed to do this now. First, we bijectively identify $\pi_0(\G f (\mathfrak{b}))$ with a product of fundamental groups of fibers (Proposition~\ref{prp:bijectionfiberwithpoint}), and then we understand the group action $\star$ in terms of composition of loops in these fundamental groups (Proposition~\ref{prp:algebraicaction}). The resulting characterization of $\Pi_0^{-1}(\Psi)$ is summarized in Theorem~\ref{thm:computationconnectedcomponents}.

Every section $\delta \in \G f (\mathfrak{b})$ restricts to a section $\delta_i : [s_i,s_{i+1}] \rightarrow X_i^{i+1}$ of $f|_{X_i^{i+1}} \rightarrow [s_i,s_{i+1}]$ for every $i = 0,\ldots,n-1$. Let
$$\G C_i : \pi_0(\G f|_{X_i^{i+1}} (\mathfrak{b}|_{[s_i,s_{i+1}]})) \rightarrow \pi_1(X_{\partial^+(i)},\mathfrak{b}(s_{\partial^+(i)}))$$ be the section collapse map defined in \eqref{eq:sectioncollapsemap}. By Lemma~\ref{lem:sectioncollapsebijectivity}, each $\G C_i$ is a bijection. This implies the following characterization of $\pi_0(\G f(\mathfrak{b}))$ in terms of the fundamental groups of the regular fibers of $f$.

\begin{prp} \label{prp:bijectionfiberwithpoint}
The product of section collapse maps
\begin{align*}
\G C : \pi_0(\G f(\mathfrak{b}))  &\rightarrow \prod_{i=0}^{n-1} \pi_1(X_{\partial^+(i)}, \mathfrak{b}(s_{\partial^+(i)}))\\
([\delta_0],\ldots,[\delta_{n-1}]) &\mapsto (\G C_0([\delta_0]),\ldots, \G C_{n-1}([\delta_{n-1}]))
\end{align*}
is bijective.
\end{prp}
\begin{proof}
Each map $\G C_i$ for $i = 0,\ldots,n-1$ is bijective by Lemma~\ref{lem:sectioncollapsebijectivity}.
\end{proof}

There is a group action
$$\hat{\star} : \prod_{i=0}^n \pi_1(X_i,\mathfrak{b}(s_i)) \times \prod_{i=0}^{n-1} \pi_1(X_{\partial^+(i)},\mathfrak{b}(s_{\partial^+(i)})) \rightarrow \prod_{i=0}^{n-1} \pi_1(X_{\partial^+(i)},\mathfrak{b}(s_{\partial^+(i)}))$$ defined as follows. The collapse maps $C_i$ from \eqref{eq:collapsemap} induce maps
\begin{align*}
\pi_1(X_i,\mathfrak{b}(s_i)) &\rightarrow \pi_1(X_{\partial^+(i)},\mathfrak{b}(s_{\partial^+(i)})),\\
\pi_1(X_{i+1},\mathfrak{b}(s_{i+1})) &\rightarrow \pi_1(X_{\partial^+(i)},\mathfrak{b}(s_{\partial^+(i)})).
\end{align*}
The action $\hat{\star}$ is defined for $\t := (\t_0,\ldots,\t_n) \in \prod_{i=0}^n \pi_1(X_i,\mathfrak{b}(s_i))$ and $\eta := (\eta_0,\ldots,\eta_{n-1}) \in \prod_{i=0}^{n-1} \pi_1(X_{\partial^+(i)},\mathfrak{b}(s_{\partial^+(i)}))$ by
\begin{equation*} \label{eq:actionononproductofcolimits}
\t  \,\,\hat{\star}\,\, \eta := (C_0(\t_0) \cdot \eta_0 \cdot C_0(\t_1)^{-1}, \ldots, C_{n-1}(\t_{n-1}) \cdot \eta_{n-1} \cdot C_{n-1}(\t_n)^{-1}).
\end{equation*}

The actions $\star$ and $\hat{\star}$ are identified by the bijection $\G C$ from Proposition~\ref{prp:bijectionfiberwithpoint}, as we now prove.

\begin{prp} \label{prp:algebraicaction}
For $\t \in \prod_{i=0}^n \pi_1(X_i,\mathfrak{b}(s_i))$ and $\delta \in \G f(\mathfrak{b})$,
$$\G C(\t \star [\delta]) = \t  \,\,\hat{\star}\,\, \G C([\delta]).$$
\end{prp}
\begin{proof}
Write $\t = (\t_0,\ldots,\t_n)$ and $[\delta] = ([\delta_0],\ldots,[\delta_{n-1}])$. By definition of $\G C$ and $\star$, we have
$$\G C(\t \star [\delta]) = \big (\G C_0(\t_0 \#_- [\delta_0] \#_+ \t_1^{-1}), \ldots, \G C_{n-1}(\t_{n-1} \#_- [\delta_{n-1}] \#_+ \t_n^{-1}) \big ),$$
which by Lemma~\ref{lem:relationshipsplicingcollapse} is equal to
$$\big ( C_0(\t_0) \cdot \G C_0([\delta_0]) \cdot C_0( \t_1)^{-1}, \ldots, C_{n-1}(\t_{n-1}) \cdot \G C_{n-1}([\delta_{n-1}]) \cdot C_{n-1}(\t_n)^{-1} \big ).$$
The above expression is equal to $\t  \,\,\hat{\star}\,\, \G C([\delta])$ by definition of $\hat{\star}$ and $\G C$.
\end{proof}

The result of the above discussion is the following theorem.

\begin{thm} \label{thm:computationconnectedcomponents}
Let $f : X \rightarrow I = [0,1]$ be a tame function, $ZX$ the associated diagram \eqref{eq:zigzagX} of regular fibers, and $\G f$ the space of sections. Then the natural map $\Pi_0 : \pi_0(\G f) \rightarrow \varprojlim \pi_0(ZX)$ is surjective.

The fibers of $\Pi_0$ are characterized as follows. Let $\Psi \in \varprojlim \pi_0(ZX)$ and $\mathfrak{b} \in \G f$ such that $\Pi_0(\mathfrak{b}) = \Psi$. Then $\Pi_0^{-1}(\Psi)$ is naturally in bijection with the orbits of the action $\hat{\star}$ of the group $\prod_{i=0}^n \pi_1(X_i,\mathfrak{b}(s_i))$ on the set $\prod_{i=0}^{n-1} \pi_1(X_{\partial^+(i)}, \mathfrak{b}(s_{\partial^+(i)}))$.
\hfill$\square$
\end{thm}

We now state a similar theorem when the domain of $f$ is $S^1$ instead of $I = [0,1]$. The proof is essentially the same.

\begin{thm} \label{thm:computationcomponentsS1}
Let $X$ be a compact manifold with boundary and $f : X \rightarrow S^1$ a smooth function that is submersive and such that its restriction to the boundary $f|_{\partial X} : \partial X \rightarrow S^1$  has isolated critical points with distinct critical values. Define the diagram $ZX$ as in \eqref{eq:zigzagX} with the additional identity map $X_0 = X_n$. Define the action $\hat{\star}$ as above for group elements satisfying $\t_0 = \t_n$. Then the statements in Theorem~\ref{thm:computationconnectedcomponents} hold.
\hfill$\square$
\end{thm}

\section{Application: Evasion paths in mobile sensor networks} \label{sec:theevasionpathproblem}

Given a collection of continuous sensors $\cS = \{ \g : [0,1] \rightarrow \cD \}$ moving in a bounded domain $\cD \subset \mathbb{R}^d$, an {\bf evasion path} is a continuous intruder $\delta : [0,1] \rightarrow \cD$ that avoids detection by the sensors for the whole time interval $I = [0,1]$. Suppose each sensor $\g$ observes a ball of fixed radius in $\mathbb{R}^d$ centered at $\g(t)$ at every $t \in I$, and let $C_t$ be the time-varying union of these sensor balls. Then the intruder $\delta$ avoids detection if $\delta(t)$ is not in $C_t$ for all $t \in I$.

The {\bf sensor ball evasion path problem} (see \textsection\ref{subsec:priorwork} for details) asks for a criterion that determines whether or not an evasion path exists and that is based on the least amount of sensed information possible; for example, a common assumption is that the sensors only detect other nearby sensors and intruders; in particular, they do not know their coordinates in $\mathbb{R}^d$. We ask to understand the homotopy type of the space $\cE$ of evasion paths. The existence problem simply asks if $\cE$ is empty or not.

We consider an idealized version of the evasion path problem, called the {\bf smoothed evasion path problem}, where we smooth the time-varying covered region $C_t$ into a smooth cobordism of manifolds with boundary; see \textsection\ref{subsec:idealized} for the precise description. In our main result (Theorem~\ref{thm:main}), we establish a necessary and sufficient condition for existence of an evasion path based on the time-varying $d-1$ homology of the covered region $C_t$ and the time-varying cup-product on cohomology $H^0$ of its boundary, and moreover we provide a lower bound on the number of connected components $\pi_0(\cE)$ of the space of evasion paths $\cE$. In the preliminary Corollary~\ref{cor:discretization}, we provide a full computation of $\pi_0(\cE)$ in terms of the time-varying $\pi_0$ and $\pi_1$ of the uncovered region $X_t = \cD \setminus C_t$, following from Theorem~\ref{thm:computationconnectedcomponents}. In the case that $C_t$ is connected for all $t \in I$, Theorem~\ref{thm:main} has the simpler form Corollary~\ref{cor:connected} that requires only the cup product on $H^0$ of the boundary of $C_t$.

Our results are the first to compute more about the space $\cE$ than whether or not it is empty.

We recall the sensor ball evasion path problem in detail and describe prior work on evasion problems in \textsection\ref{subsec:priorwork}. Then we introduce the smoothed evasion problem in \textsection\ref{subsec:idealized}, and explain our results in \textsection\ref{sec:results}.

\subsection{Prior work} \label{subsec:priorwork}
We describe the sensor ball evasion path problem in more detail. Then we explain prior results on this problem and a general evasion problem studied in \cite{MR3763757}.

Fix some sensing radius $r > 0$ and say that each sensor $\g \in \cS$ can detect objects within the closed ball
$$B_{\g(t)} = \{ x \in \mathbb{R}^d \,\, | \,\, |\g(t) - x| \leq r \}$$
at all times $t \in I = [0,1]$. The time-$t$ covered region
\begin{equation} \label{eq:sensorballunion}
C_t = \bigcup_{\g \in \cS} B_{\g(t)} \subset \mathbb{R}^d
\end{equation}
is the union of the sensor balls, and it is homotopy equivalent to the \v{C}ech complex of the covering of $C_t$ by the sensor balls. The time-$t$ uncovered region is the complement
$$X_t = \cD \setminus C_t,$$
where $\cD \subset \mathbb{R}^d$ is a bounded domain that is homeomorphic to a ball.

Assume that the boundary $\partial \cD$ is covered by a collection of immobile fence sensors $F\cS \subset \cS$, meaning $\g(t) \in \mathbb{R}^d$ is constant for $\g \in F\cS$, and that the union of balls $B_{\g(t)}$ for $\g \in F\cS$ covers the boundary $\partial \cD$ and is homotopy equivalent to $\partial \cD$. In particular, this ensures that an intruder $\delta$ can never escape from the domain $\cD$.

We define the covered region
\begin{equation} \label{eq:coveredregion}
C := \bigcup_{t \in I} C_t \times \{t\} \subset \mathbb{R}^d \times I
\end{equation}
and the uncovered region
$$X := (\cD \times I) \setminus C = \bigcup_{t \in I} X_t \times \{t\}.$$
Then an evasion path is equivalent to a continuous section $\delta : I \rightarrow X$ of the projection $\rho_{X} : X \rightarrow I$, i.e.,\ $\rho_{X} \circ \delta = id_I$.

The sensor ball evasion path problem asks for a criterion for existence of an evasion path.

This problem was first stated and studied in \cite{deSilvaGhristEvasionsFence} by de Silva and Ghrist in dimension $d = 2$. Provided that the time-varying \v{C}ech complex changes only at finitely many times, the authors establish a necessary condition for existence of an evasion path which states that the connecting homomorphism on a relative homology group with respect to the fence of the covered region must vanish.

In \cite{EvasionAdamsCarlsson}, Adams and Carlsson consider the same evasion problem in arbitrary dimension $d \geq 2$, and they establish a necessary condition for the existence of an evasion path based on zigzag persistent homology \cite{MR2657946} of the time-varying \v{C}ech complex. They also explain how this result is equivalent to a generalization of de Silva and Ghrist's result to the general case $d \geq 2$.

In both de Silva-Ghrist and Adams-Carlsson, the necessary condition for existence of an evasion path is determined and easily computable from the overlap information of sensor balls. Essentially, sensors only need to know which other sensors are nearby; they don't need to know their coordinates in $\mathbb{R}^d$.

Adams-Carlsson also study smarter sensors in the plane $\mathbb{R}^2$ that know local distances to nearby sensors and the natural counterclockwise ordering on nearby sensors. For these smarter sensors in $\mathbb{R}^2$, they derive a necessary and sufficient condition for existence of an evasion path as well as an algorithm to compute it, under the assumption that the covered region $C_t$ is connected for all $t \in I$. They show that this connectedness assumption is necessary for their methods.

In \cite{MR3763757}, Ghrist and Krishnan define positive homology and cohomology for directed spaces over $\mathbb{R}^q$, and they compute it using sheaf-theoretic techniques. They consider a general type of evasion problem where the covered region is a pro-object in a category of smooth compact cobordisms. They derive a criterion on positive cohomology that is necessary and sufficient for existence of an evasion path, under the assumption that $C_t$ is connected.

\subsection{The smoothed evasion path problem} \label{subsec:idealized}

In this section we describe the precise version of the evasion path problem considered in this paper and indicate how our setting arises as a limiting case of the sensor ball evasion path problem as the number and density of the sensors becomes large.

Fix a dimension\footnote{We assume $d \geq 2$ because the $d = 0,1$ cases are trivial under the conditions of our Theorem~\ref{thm:main} and would require slightly modified arguments to state as part of the theorem. See Remark~\ref{rmk:mainlowdims} for the $d = 0,1$ cases.} $d \geq 2$ and let $\cD \subset \mathbb{R}^d$ be a smoothly embedded closed $d$-dimensional ball. The {\bf covered region} $C \subset \cD \times I$, where $I = [0,1]$, is any subset with the following properties:

\begin{itemize}
\item The time $0$ and $1$ covered regions, given by
$$C_i = C \cap (\mathbb{R}^d \times \{i\})$$
for $i = 0,1$, are smooth compact codimension-$0$ submanifolds $C_i \subset \mathbb{R}^d$ with smooth boundary $\partial C_i$.

\item The full covered region $C \subset \mathbb{R}^d \times I$ is a smooth embedded compact cobordism (Definition~\ref{dfn:cobordism})
between $C_0$ and $C_1.$ In particular, $C$ has boundary
$$\partial C = C_0 \cup \partial \cup C_1$$
where $\partial$ is a $d$-dimensional manifold with boundary $\partial C_0 \sqcup \partial C_1$ that intersects $C_i$ along its boundary
$$\partial \cap C_i = \partial C_i$$
for $i = 0,1$.

\item The boundary $\partial$ contains $\partial \cD \times I$ as a connected component.

\item The critical points of the projection $\rho_{\partial} : \partial \rightarrow I$ are isolated and have distinct critical values contained in $(0,1)$.
\end{itemize}

\begin{rmk}
Heuristically, when the covered region $C$ is given by the union of sensor balls over all times $t \in I$ as in \eqref{eq:sensorballunion} and \eqref{eq:coveredregion}, the smoothed evasion path problem arises from a small perturbation of $C$ inside $\mathbb{R}^d \times I$ that smooths out the time-varying union of sensor balls into a manifold. The projection $C \rightarrow I$ does not have any critical points since $C$ is codimension-$0$ in $\mathbb{R}^d \times I$, and generically the projection on the boundary $\rho_{\partial} : \partial \rightarrow I$ has isolated critical points with distinct critical values.

For generic sensor paths in $\mathbb{R}^d$, the boundaries of the sensors balls will intersect transversely at all but finitely many times $t \in I$. We call the transverse times the regular times, and the non-transverse times are the critical times. Heuristically, these times correspond to the regular values and the critical values, respectfully, of the function $\rho_{\partial}$ in the smoothed evasion path problem.
\end{rmk}

We define the uncovered region $X$ to be the closure of the complement $(\cD \times I) \setminus C$. More precisely, we define the {\bf essential boundary}
$$B := \partial \setminus (\partial \cD \times I),$$
and the {\bf uncovered region}
$$X :=  ((\cD \times I) \setminus C) \cup B.$$
Then $X$ is a compact cobordism (Definition~\ref{dfn:cobordism}) between the time $0$ and time $1$ uncovered regions $X_i = X \cap (\mathbb{R}^d \times \{i\})$ for $i = 0,1,$ with boundary
$$\partial X = X_0 \cup B \cup X_1.$$
Note that $B$ is the intersection
$$B = X \cap C.$$

Consider the projection $\rho : \mathbb{R}^d \times I \rightarrow I$ and the {\bf space of evasion paths}
\begin{equation} \label{eq:spaceofevasionpaths}
\cE := \{ \delta : I \rightarrow X \,\, | \,\, \rho \circ \delta = id_I \text{ and } \delta \text{ is continuous}\},
\end{equation}
i.e.,\ the space of continuous sections of the projection restricted to the uncovered region
$$\rho_{X} : X \rightarrow I.$$
Here, $\cE$ is given the compact-open topology, or equivalently the topology of the supremum norm for continuous maps $I \rightarrow \mathbb{R}^d$ with respect to the standard norm on $\mathbb{R}^d$.

Our main result (Theorem~\ref{thm:main}) establishes a necessary and sufficient condition for existence of an evasion path, and moreover provides a lower bound on the number of connected components $\pi_0(\cE)$. The required information consists of a parameterized homology of the covered region $C$, a parameterized cup-product on the essential boundary $B$, and an Alexander duality isomorphism. See also Corollary~\ref{cor:connected} for the simpler case when $C_t$ is connected for all $t$.

We now briefly discuss how one might approximate the information required in Theorem~\ref{thm:main} to compute existence of an evasion path in the sensor ball evasion path problem where the covered region $C$ is a time-varying union of sensors balls as in \eqref{eq:sensorballunion} and \eqref{eq:coveredregion}.

To compute the time-varying homology of $C$ it suffices to have the information of the time-varying \v{C}ech complex, or in other words the information of sensor ball overlaps; see \cite{EvasionAdamsCarlsson}.

Heuristically, one can use the Niyogi-Smale-Weinberger Theorem \cite{MR2383768} to compute the time-varying cup-product on cohomology $H^0$ of the essential boundary $B$ under the following additional assumptions:
\begin{itemize}
\item {\bf Directional sensing}: The sensors detect the direction in $\mathbb{R}^d$ of nearby sensors, as well as the direction of the wall $\partial \cD$ if it is nearby.
\item {\bf Dense coverage}: The sensors are $\epsilon$-densely distributed throughout $C_t$ at all regular values $t$, in the sense of \cite{MR2383768}.
\end{itemize}
Indeed, with these assumptions, we can label a sensor $\g \in \cS$ at time $t \in I$ as an {\bf essential boundary sensor} if there is a codimension-$1$ hypersurface in $\mathbb{R}^d$ containing $\g(t)$ such that all sensors in a small neighborhood of $\g(t)$ lie on the same side of the hypersurface, and such that the other side of the hypersurface is not the wall $\partial \cD$. First of all, sensors can detect this information due to the directional sensing hypothesis. Second, due to the dense coverage hypothesis, such a hypersurface exists if and only if the sensor is close to the essential boundary $B$; indeed, sensors in the interior of $C$ see other sensors in all directions, whereas those near $B$ see other sensors towards the interior of $C$ and they don't see any sensors on the other side of $B$. This same effect occurs near $\partial \cD$, yet the sensor is aware that it is an effect of the wall. The Niyogi-Smale-Weinberger Theorem \cite{MR2383768} then asserts that the \v{C}ech complex of the boundary sensors computes cohomology of $B_t = B \cap (\mathbb{R}^d \times \{t\})$ at regular times $t$.

\subsection{Connected components of the space of evasion paths} \label{sec:results}
In this section we prove the results, Theorem~\ref{thm:main}, Corollary~\ref{cor:discretization}, and Corollary~\ref{cor:connected}, about the smoothed evasion path problem introduced in \textsection\ref{subsec:idealized}.

Denote the critical values of the projection $\rho_B : B \rightarrow I$ by
$$t_1 < t_2 < \cdots < t_n$$
and choose interleaving values
$$0 = s_0 < t_1 < s_1 < t_2 < \cdots < s_{n-1} < t_n < s_n =  1.$$
We recall the zigzag discretization procedure described in \textsection \ref{subsec:zigzagdiscretization}, which we apply here to various subspaces $M \subset \mathbb{R}^d \times I.$
Consider the projection
$$\rho : \mathbb{R}^d \times I \rightarrow I$$
and its restriction
$$\rho_M := \rho|_{M} : M \rightarrow I,$$
from which we obtain subspaces of $M$ given by the preimages
\begin{align*}
M_i &:= \rho_{M}^{-1}(s_i), \text{ for } 0 \leq i \leq n,\\
M_i^{i+1} &:= \rho_{M}^{-1}([s_i,s_{i+1}]), \text{ for } 0 \leq i \leq n-1.
\end{align*}
These fit together into a zigzag diagram of topological spaces
\begin{equation*}
\widetilde{ZM} := \bigg ( M_0 \hookrightarrow M_0^1 \hookleftarrow M_1 \hookrightarrow M_1^2 \hookleftarrow \cdots \hookrightarrow M_{n-1}^n \hookleftarrow M_n \bigg )
\end{equation*}
where all maps are the natural inclusions.

Consider now the zigzag diagram $\widetilde{ZX}$ in the case that $M = X$ is the uncovered region.  Recall the abbreviated zigzag diagram $$ZX := (X_0 \leftrightarrow X_1 \leftrightarrow \cdots \leftrightarrow X_n)$$
defined in \eqref{eq:zigzagX}, which is obtained roughly by inverting homotopy equivalences. Note that
$$\varprojlim \pi_0(\widetilde{ZX}) = \varprojlim \pi_0(ZX).$$
Theorem~\ref{thm:computationconnectedcomponents} applied to the projection $\rho_X : X \rightarrow I$ provides the following.

\begin{cor} \label{cor:discretization}
There is a surjection
$$\Pi_0 : \pi_0(\cE) \rightarrow \varprojlim \pi_0(ZX) = \varprojlim \pi_0(\widetilde{ZX}).$$
In the notation of \textsection \ref{sec:sectionsoftamefunction}, the fibers of $\Pi_0$ are characterized as follows. Let $\Psi \in \varprojlim \pi_0(ZX)$ and $\mathfrak{b} \in \cE$ such that $\Pi_0(\mathfrak{b}) = \Psi$. Then the fiber $\Pi_0^{-1}(\Psi)$ is naturally in bijection with the orbits of the action $\hat{\star}$ of the group $\prod_{i=0}^n \pi_1(X_i,\mathfrak{b}(s_i))$ on the set $\prod_{i=0}^{n-1} \pi_1(X_{\partial^+(i)}, \mathfrak{b}(s_{\partial^+(i)}))$.
\end{cor}
\begin{proof}
The space $\cE$, defined in \eqref{eq:spaceofevasionpaths}, is the space of continuous sections of the projection $\rho_{X} : X \rightarrow I$. Since $X \subset \mathbb{R}^d \times I$ is codimension-$0$, the projection $\rho_{X}$ is submersive. Since the projection on the boundary $\rho_{\partial} : \partial \rightarrow I$ has isolated critical points with distinct critical values in $(0,1)$ by hypothesis, the same is true for its restriction to the essential boundary $B \subset \partial$, and hence $\rho_{X}$ is a tame function (Definition~\ref{dfn:tamefunction}). Hence the result follows from Theorem~\ref{thm:computationconnectedcomponents}.
\end{proof}

\begin{rmk} \label{rmk:applyresults}
Corollary~\ref{cor:discretization}, applied to the examples presented in Figure~\ref{fig:examples} produces the same full computation of $\pi_0(\cE)$ as Theorem~\ref{thm:pinbaby}, namely exactly that described in Remark~\ref{rmk:applytheoremtoexamples}.

Our main result, Theorem~\ref{thm:main}, provides only lower bounds on the cardinality of $\pi_0(\cE)$ and determines if it is empty or not, however it needs as input only (co)homological information about the covered region $C$ and the essential boundary $B$. This is important for applications in which the topology of $C$ and the boundary can be detected from the sensor information but the topology of $X$ may be harder to determine since this is defined as the region where there are no sensors. To improve this to a full computation of $\pi_0(\cE)$, one would need to compute the fiberwise fundamental groups and maps between them used in Corollary~\ref{cor:discretization} from homological information. This can be addressed by a fiberwise version of the unstable Adams spectral sequence of Bousfield-Kan, currently under development by Wyatt Mackey.
\end{rmk}

The goal now is to compute $\pi_0(\widetilde{ZX})$ in terms of homological information about the covered region $C$ and the boundary $B$. In principle, we have access to this homological information in applications with moving sensors (see the end of \textsection \ref{subsec:idealized}).

Let $k$ be a field. In view of Proposition~\ref{prp:dualizeH0algebra} below, to compute $\pi_0(\widetilde{ZX})$ it suffices to compute the cup product structure on $H^0(\widetilde{ZX};k)$.

Let $Hom_{k-algebra}(-,k)$ be the functor that takes a $k$-algebra $R$ to the set of $k$-algebra homomorphisms $R \rightarrow k$. Then applying $Hom_{k-algebra}(-,k)$ to the diagram $H^0(\widetilde{ZX};k)$ yields a zigzag diagram of sets.

\begin{prp} \label{prp:dualizeH0algebra}
There is an isomorphism of zigzag diagrams of sets
$$\pi_0(\widetilde{ZX}) \cong Hom_{k-algebra}(H^0(\widetilde{ZX};k),k),$$
or in other words a commutative diagram 
\[
  \begin{tikzcd}
 \pi_0({X}_0) \arrow{r} \arrow{d}   &  \pi_0({X}_0^1)   \arrow{d}   & \arrow{l} \cdots \arrow{r} &\pi({X}_{n-1}^n)  \arrow{d} & \arrow{l} \pi_0({X}_n) \arrow{d} \\
H^0({X}_0;k)^{\vee}  \arrow{r} & H^0({X}_0^1;k)^{\vee}  & \arrow{l} \cdots \arrow{r} &H^0({X}_{n-1}^n;k)^{\vee}  & \arrow{l} H^0({X}_n;k)^{\vee}
\end{tikzcd}
\]
with all vertical arrows bijections and where $H^0(-,k)^{\vee} := Hom_{k-algebra}(H^0(-;k),k)$ is the set of $k$-algebra morphisms $H^0(-;k) \rightarrow k$ with respect to the cup product on $H^0$.
\end{prp}
\begin{proof}
For any topological space $A$, there is a bijection of sets 
\begin{align*}
\alpha : \pi_0(A) &\rightarrow Hom_{k-algebra}(H^0(A;k),k)\\
[a] &\mapsto \bigg ((\varphi \in H^0(A;k)) \mapsto \varphi(a) \in k \bigg ).
\end{align*}
Moreover, $\alpha$ is a natural isomorphism from the functor $\pi_0(-)$ to the functor $Hom_{k-algebra}(H^0(-;k),k)$. Both of these are functors from the category of topological spaces to the category of sets. Hence $\alpha$ induces the claimed isomorphism of zigzag diagrams.
\end{proof}

It remains to compute $H^0(\widetilde{ZX};k)$. The inclusion map $\i_B : B \rightarrow X$ induces a map of zigzag diagrams of $k$-algebras $\i_B^* : H^0(\widetilde{ZX};k) \rightarrow H^0(\widetilde{ZB};k)$. In the simplest situation, when the time-$t$ covered region $C \cap (\mathbb{R}^d \times \{t\})$ is connected for all $t \in I$, the map $\i_B^*$ is an isomorphism by Proposition~\ref{prp:computationintermsofcoveredandboundary} below. The final result in this case is Corollary~\ref{cor:connected}.

In Theorem~\ref{thm:main} below, we explain how to compute $H^0(\widetilde{ZX};k)$ in the case that $C \cap (\mathbb{R}^d \times \{t\})$ is not necessarily connected. Let
$$B^c = (\cD \times I) \setminus B$$
be the complement of the essential boundary. Recall the notation for the regular level sets $B_i = B \cap (\mathbb{R}^d \times \{s_i\}), B^c_i = B^c \cap (\mathbb{R}^d \times \{s_i\})$, and $C_i = C \cap (\mathbb{R}^d \times \{s_i\})$ for $0 \leq i \leq n$. For each $0 \leq i \leq n$, there is a composite map
$$H_{d-1}(C_i;k) \xrightarrow{{\i_{C_i}}_*} H_{d-1}(B^c_i;k) \xrightarrow{\alpha_i} H^0(B_i;k)$$
where ${\i_{C_i}}_*$ is induced by the inclusion\footnote{For a manifold $C$ with boundary, the inclusion $\text{Interior}(C) \subset C$ is a homotopy equivalence.} $\i_{C_i} : \text{Interior}(C_i) \rightarrow B_i^c$ and $\alpha_i$ is Alexander duality (see Remark~\ref{rmk:Alexanderduality}).

See Example~\ref{ex:homologicaltheorem} below for the computation of $H^0(\widetilde{ZX};k)$ in example (a) from Figure~\ref{fig:examples}.

\begin{thm} \label{thm:main}
There is a surjection $\pi_0(\cE) \rightarrow \varprojlim Hom_{k-algebra}( H^0(\widetilde{ZX};k), k)$. In particular, an evasion path exists (i.e.\ $\cE$ is nonempty) if and only if $\varprojlim Hom_{k-algebra}( H^0(\widetilde{ZX};k), k )$ is nonempty, and the cardinality of $\pi_0(\cE)$ is bounded from below by the cardinality of the inverse limit.

Assume that the projection $C \rightarrow I$ does not have any local maxima or local minima except over $0,1 \in I$. Then the following information determines the zigzag diagram of $k$-algebras $H^0(\widetilde{ZX};k)$ up to isomorphism:
\begin{enumerate}
\item The zigzag diagram of $k$-algebras $H^0(\widetilde{ZB};k)$,
\item The images $im(\alpha_i \circ {\i_{C_i}}_*) \subset H^0(B_i;k)$ for $0 \leq i \leq n$,
\item The type (N or D) of every boundary critical point of $C \rightarrow I$.
\end{enumerate}
\end{thm}

\begin{rmk}
The assumption and the given information $(i)-(iii)$ in Theorem~\ref{thm:main} are justified when interpreted in the context of a mobile sensor network.

Indeed, a local maximum of $C \rightarrow I$ can only occur if a sensor disappears. A local minimum can occur only if a new sensor is created. So, the assumption that $C \rightarrow I$ does not have local minima or maxima at intermediate times $t \in (0,1)$ means that the result applies to mobile sensor networks in which sensors do not get added or removed during the time interval of interest.

The assumed information in (i)-(iii) is all determined by the boundary $B$ and the covered region $C$. Hence, it is reasonable to assume that the sensors, which cover $C$ and have boundary $B$, can record this information; see the discussion of the Niyogi-Smale-Weinberger theorem at the end of \textsection\ref{subsec:idealized} for more detail. In particular, we do not require a priori information about the uncovered region $X$; indeed, the theorem computes the required cohomological information $H^0(\widetilde{ZX};k)$ from the given information in $(i)-(iii)$. 
\end{rmk}

\begin{proof}[Proof of Theorem~\ref{thm:main}]

The claimed surjection
$$\pi_0(\cE) \rightarrow \varprojlim Hom_{k-algebra}( H^0(\widetilde{ZX};k), k)$$
is the composition of the surjection $\Pi_0 : \pi_0(\cE) \rightarrow \varprojlim \pi_0(\widetilde{ZX})$ from Corollary~\ref{cor:discretization} and the isomorphism $\varprojlim \pi_0(\widetilde{ZX}) \cong \varprojlim Hom_{k-algebra}( H^0(\widetilde{ZX};k), k)$ from Proposition~\ref{prp:dualizeH0algebra}. It remains to compute $H^0(\widetilde{ZX};k)$ from (i)-(iii) under the assumption that $C \rightarrow I$ does not have any local maxima or minima at intermediate times $t \in (0,1)$.

The inclusion
$$\i_B : B \rightarrow X$$
induces a map of zigzag diagrams of $k$-algebras
$$\i_B^* : H^0(\widetilde{ZX};k) \rightarrow H^0(\widetilde{ZB};k).$$

\begin{prp} \label{prp:computationintermsofcoveredandboundary}
The map $\i_B^*$ is injective. If the time-$t$ covered region $C \cap (\mathbb{R}^d \times \{t\})$ is connected for all $t \in I$, then $\i_B^*$ is an isomorphism.
\end{prp}
\begin{proof}
The claim that $\i_B^*$ is injective means that the restriction maps
\begin{align*}
\i_i^* : H^0(X_i)& \rightarrow H^0(B_i),\\
{\i_i^{i+1}}^* : H^0(X_i^{i+1}) &\rightarrow H^0(B_i^{i+1})
\end{align*}
induced by the inclusions $B_i \subset X_i$ and $B_i^{i+1} \subset X_i^{i+1}$ are injective for all $i$.

We first consider the restriction maps $\i_i^*$. By definition, we have the covering $X_i \cup C_i = \cD$ and the intersection along their common manifold boundary is $X_i \cap C_i = B_i$. Hence there is a Mayer-Vietoris sequence
$$\tilde{H}^0(\cD) \rightarrow \tilde{H}^0(X_i) \oplus \tilde{H}^0(C_i) \rightarrow \tilde{H}^0(B_i) \rightarrow \tilde{H}^{1}(\cD).$$
Since $\cD$ is homeomorphic to a ball and hence contractible, it follows that the map
$\tilde{H}^0(X_i) \oplus \tilde{H}^0(C_i) \rightarrow \tilde{H}^0(B_i)$ is an isomorphism. Restricted to $\tilde{H}^0(X_i) \oplus \{0\}$, this is exactly the restriction map on reduced cohomology induced by $B_i \subset X_i$, and it is injective. Hence the induced map on unreduced cohomology, which is $\i_i^*$, is also injective, as claimed.

Under the additional hypothesis that the time-$t$ covered region $C \cap (\mathbb{R}^d \times \{t\})$ is connected for all $t \in I$, we have that $C_i = C \cap (\mathbb{R}^d \times \{s_i\})$ is connected and so $\tilde{H}^0(C_i) = 0$, in which case $\i_i^*$ is an isomorphism, as claimed.

We now return to the general case and analyze the restriction maps ${\i_i^{i+1}}^*
$. The argument that these are injective is similar. By definition, we have the covering $X_i^{i+1} \cup C_i^{i+1} = \cD \times [s_i,s_{i+1}]$ with intersection along the common boundary $X_i^{i+1} \cap C_i^{i+1} = B_i^{i+1}$. Hence there is a Mayer-Vietoris sequence
$$\tilde{H}^0(\cD \times [s_i,s_{i+1}]) \rightarrow \tilde{H}^0(X_i^{i+1}) \oplus \tilde{H}^0(C_i^{i+1}) \rightarrow \tilde{H}^0(B_i^{i+1}) \rightarrow \tilde{H}^{1}(\cD \times [s_i \times s_{i+1}]).$$
Hence the map $\tilde{H}^0(X_i^{i+1}) \oplus \tilde{H}^0(C_i^{i+1}) \rightarrow \tilde{H}^0(B_i^{i+1})$ is an isomorphism, which implies that the restriction map ${\i_i^{i+1}}^*
$ is injective, as claimed.

Under the additional hypothesis that the time-$t$ covered region $C \cap (\mathbb{R}^d \times \{t\})$ is connected for all $t \in I$, it follows from Lemma~\ref{lem:defretracttoregvalues} applied to the tame function $C_i^{i+1} \rightarrow I$ that the cobordism $C_i^{i+1}$ is homotopy equivalent to either $C_i$ or $C_{i+1}$, hence is connected, for all $i$. Thus $\tilde{H}^0(C_i^{i+1}) = 0$ and so ${\i_i^{i+1}}^*$ is an isomorphism, as claimed. This completes the proof of the proposition.
\end{proof}

Proposition~\ref{prp:computationintermsofcoveredandboundary} implies that $H^0(\widetilde{ZX};k)$ is isomorphic to the image $\text{im}(\i_B^*)$ of $\i_B^*$ a zigzag diagram of $k$-algebras, where the zigzag $k$-algebra structure on $\text{im}(\i_B^*)$ is induced by the cup-product on $H^0(\widetilde{ZB} ;k)$. Hence, it suffices to compute
$$\text{im}(\i_{B_i}^*) \subset H^0(B_i;k) \,\,\text{ and }\,\, \text{im}(\i_{B_i^{i+1}}^*) \subset H^0(B_i^{i+1};k)$$
as $k$-vector spaces for all $i$. Indeed, the $k$-algebra structures on these images are induced by the cup products on $H^0(B_i;k)$ and $H^0(B_i^{i+1};k)$, respectively, which are given by assumption $(i)$.

For each $i$, there is a commutative diagram of $k$-vector spaces
\begin{equation*} \label{eq:parameterizedalexanderevasion}
  \begin{tikzcd}
 H^0(X_i; k) \arrow{r}{\i_{B_i}^*}\arrow{d}{\cA_i}   &  H^0(B_i; k)\arrow{d}{\alpha_i^{-1}}    \\
H_{d-1}(C_i; k)  \arrow{r}{{\i_{C_i}}_*} & H_{d-1}(B_i^c; k),
\end{tikzcd}
\end{equation*}
where the top map is the restriction map on cohomology induced by the inclusion $\i_{B_i} : B_i \hookrightarrow X_i$, the bottom map is the pushforward on homology $H_{d-1}(-;k)$ induced by the inclusion $\i_{C_i} : Interior(C_i) \rightarrow B^c$, and the vertical maps $\cA_i$ and $\alpha_i$ are Alexander duality isomorphisms (see Remark~\ref{rmk:Alexanderduality}). This diagram commutes, and hence
$$im(\i^*_{B_i}) = im(\alpha_i \circ {\i_{C_i}}_*).$$
The right hand side is given by assumption $(ii)$.

\begin{rmk} \label{rmk:Alexanderduality}
In the Alexander duality isomorphisms $\cA_i$ and $\alpha_i$ we have unreduced cohomology $H^0$ instead of reduced cohomology $\tilde{H}^0$ because $X_i \cup C_i = \cD$ and $B_i \cup B_i^c = \cD$ are missing the connected subset $\mathbb{R}^d \setminus \cD$ of $\mathbb{R}^d$ and moreover $\mathbb{R}^d \setminus \cD$ is disjoint from both $X_i$ and $B_i$.
\end{rmk}

It remains to compute $im(\i^*_{B_i^{i+1}}).$ Consider the following commutative diagram, which is the relevant part of the map $H^0(\widetilde{ZX}) \rightarrow H^0(\widetilde{ZB})$ induced by the inclusion $B \hookrightarrow X$,

\[
  \begin{tikzcd}
 H^0(X_i)      \arrow[hook]{d}{\i_{B_i}^*} & \arrow{l}{\phi} H^0(X_i^{i+1}) \arrow[hook]{d}{\i_{B_i^{i+1}}^*}  \arrow{r}  &  H^0(X_{i+1}) \arrow[hook]{d}{\i_{B_{i+1}}^*}  \\
 H^0(B_i)    &  \arrow{l}{\varphi}  H^0(B_i^{i+1})  \arrow{r}  &  H^0(B_{i+1}).
 \end{tikzcd}
\]

All vertical maps are injective by Proposition~\ref{prp:computationintermsofcoveredandboundary}.

There is a single critical point $p \in B_i^{i+1}$ of the projection $B \rightarrow I$, and by assumption $(iii)$ we know if $p$ is type-$N$ or type-$D$ with respect to $C \rightarrow I$. Assume now that it is type-$N$; the argument in the type-$D$ case is symmetric. Since $p$ is type-$N$ with respect to the projection $C \rightarrow I$, it is type-$D$ with respect to the projection $X \rightarrow I$. Hence $X_i^{i+1}$ deformation retracts onto $X_i$ by Lemma~\ref{lem:defretracttoregvalues}, which implies that $\phi$ is an isomorphism.

We claim that $\varphi$ is injective. Indeed, any nontrivial kernel of $\varphi$ implies the existence of a component $A$ of the cobordism $B_i^{i+1}$ whose boundary $\partial A$ is disjoint from $B_i$ and hence lies entirely in  $B_{i+1}$. Hence, since $p$ is the only critical point in $B_i^{i+1}$, it must lie on $A$, since any component of $B_i^{i+1}$ without a critical point is a trivial cobordism. It follows that $p$ is a local minimum of $B \rightarrow I$. Since $p$ is type-$N$ with respect to $C \rightarrow I$, the normal vector to $B$ at $p$ pointing into the interior of $C$ is pointing positively along $I$ and hence $p$ is also a local minimum of $C \rightarrow I$. This contradicts the assumption that there are no local minima of $C \rightarrow I$ away from $0$ and $1$, proving the claim that $\varphi$ is injective.

To complete the proof, recall that in the commutative diagram above all vertical maps are injective, $\phi$ is an isomorphism, and $\varphi$ is injective. It follows that $im(\i^*_{B_i^{i+1}}) = \varphi^{-1}(im(\i^*_{B_i}))$. Since we computed $im(\i^*_{B_i})$ above and since the map $\varphi$ is part of the given data in assumption $(i)$, this completes the computation of $im(\i^*_{B_i^{i+1}})$ and the proof of Theorem~\ref{thm:main}.
\end{proof}

\begin{ex} \label{ex:homologicaltheorem}
We use the method in Theorem~\ref{thm:main} to compute $H^0(\widetilde{ZX};k)$ in example (a) from Figure~\ref{fig:examples}, providing the lower bound $|\pi_0(\cE)| \geq |\varprojlim Hom_{k-algebra}( H^0(\widetilde{ZX};k), k)| =  2$.

The zigzag diagram of spaces $\widetilde{ZB} = (B_0 \hookrightarrow B \hookleftarrow B_1)$ is a circle $B_0 = S^1$ and two circles $B_1 = S^1 \sqcup S^1$ included as the boundary of the pair of pants $B$. Hence we have the zigzag diagram of $k$-algebras
\begin{align*}
H^0(\widetilde{ZB}; k) = (k \hookleftarrow k \hookrightarrow k \oplus k).
\end{align*}
Since the time-$t$ covered region is connected for all $t \in I$, the map $\i_B^* : H^0(\widetilde{ZX};k) \rightarrow H^0(\widetilde{ZB};k)$ is an isomorphism by Proposition~\ref{prp:computationintermsofcoveredandboundary}. Hence $\varprojlim Hom_{k-algebra}( H^0(\widetilde{ZX};k), k)$ has cardinality $2$.
\end{ex}

\begin{rmk} \label{rmk:mainlowdims}
Theorem~\ref{thm:main} is stated and proved in dimensions $d \geq 2$. Here we explain the $d = 0,1$ cases. For $d = 0$, the only possibility is $C = \mathbb{R}^0 \times I$ and the uncovered region is empty $X = \emptyset$ so there are no evasion paths $\cE = \emptyset.$ In the $d = 1$ case, the assumption that the projection $C \rightarrow I$ does not have any local minima or maxima away from $0,1 \in I$ implies that the cardinality of $\pi_0(\cE)$ is equal to the number of connected components of the time-$0$ uncovered region $\pi_0(X_0)$ minus the number of type D boundary critical points of $X \rightarrow I$. This is equal to $|\pi_0(\cE)| = |\pi_0(C_0)| - 1 - \#(\text{type N critical points of } C \rightarrow I).$
\end{rmk}

\bibliographystyle{amsplain}
\bibliography{../../references}
\end{document}